\documentclass[review]{elsarticle}

\usepackage{amsmath}
\usepackage{amsfonts}
\usepackage{amsthm}
\usepackage{hyperref}
\newtheorem{theorem}{Theorem}
\newtheorem{definition}{Definition}
\newtheorem{proposition}{Proposition}
\newtheorem{lemma}{Lemma}

\journal{Journal of \LaTeX\ Templates}

\begin{document}
	
\begin{frontmatter}

\title{On some formulas for the $k$-analogue of Appell functions and generating relations via $k$-fractional derivative \tnoteref{mytitlenote}}

\author[mymainaddress]{\"{O}vg\"{u} G\"{u}rel Y\i lmaz}

\author[mysecondaryaddress]{Rabia Akta\c{s}}

\author[mysecondaryaddress]{Fatma Ta\c{s}delen}

\address[mymainaddress]{Department of Mathematics, Recep Tayyip Erdogan
	University, Rize, Turkey}
\address[mysecondaryaddress]{Department of Mathematics, Faculty of
	Science, Ankara University, 06100, Tandogan, Ankara, Turkey.}

\begin{abstract}
Our present investigation is mainly based on the $k$-hypergeometric
functions which are constructed by making use of the Pochhammer $k$-symbol 
\cite{Diaz} which are one of the vital generalization of hypergeometric
functions. We introduce $k$-analogues of $F_{2}\ $and $F_{3}$ Appell functions
denoted by the symbols $F_{2,k}\ $and $F_{3,k}\ $respectively, just like
Mubeen et al. did for $F_{1}$ in 2015 \cite{Mubeen6}. Meanwhile, we prove
some main properties namely integral representations, transformation
formulas and some reduction formulas which help us to have relations between
not only $k$-Appell functions but also $k$-hypergeometric functions.
Finally, employing the theory of Riemann Liouville $k$-fractional derivative 
\cite{Rahman} and using the relations which we consider in this paper, we
acquire linear and bilinear generating relations for $k$-analogue of
hypergeometric functions and Appell functions.
\end{abstract}

\begin{keyword}
\texttt{k-Gamma function, k-Beta function, Pochhammer symbol, Hypergeometric function, Appell functions, Integral representation, Reduction and transformation formula, Fractional derivative, Generating function.}
\MSC[2010] 	26A33\sep 33B15\sep 33C05\sep 33C65\sep 
\end{keyword}

\end{frontmatter}


\section{Introduction}

Special functions, with its diverse sub-branches, are very wide field of
study and are used not only in various fields of mathematics but also in the
solutions of important problems in many disciplines of science such as
physics, chemistry and biology. This subject is powerful to make sense of
uncertain questions especially in physical problems so it encourages many
people for notable improvements on this matter. As in other sciences,
remarkable problems are still discussed in many disciplines and more general
results are tried to be obtained.

Generalized hypergeometric functions which are one of these studies in
special functions \cite{Slater, Srivastava} are defined by 
\begin{equation}
_{p}F_{q}\left[ 
\begin{array}{c}
\alpha _{1},\ \alpha _{2},...,\alpha _{p} \\ 
\beta _{1},\ \beta _{2},...,\beta _{q}%
\end{array}%
;x\right] =\sum_{n=0}^{\infty }\frac{\left( \alpha _{1}\right)
	_{n}\left( \alpha _{2}\right) _{n}...\left( \alpha _{p}\right) _{n}}{\left(
	\beta _{1}\right) _{n}\left( \beta _{2}\right) _{n}...\left( \beta
	_{q}\right) _{n}}\frac{x^{n}}{n!}  \label{ghf}
\end{equation}%
where $\alpha _{1},\ \alpha _{2},....,\alpha _{p},\ \beta _{1},\ \beta
_{2},...,\beta _{q},\ x\in 
\mathbb{C}
\ $and $\beta _{1},\ \beta _{2},...,\beta _{q}$ neither zero nor a negative
integers.

Here $\left( \lambda \right) _{n}$ is the Pochhammer symbol defined by
\begin{equation}
\left( \lambda \right) _{n}=\left\{ 
\begin{array}{c}
\lambda \left( \lambda +1\right) ...\left( \lambda +n-1\right) \ ;\ \ \
n\geq 1\ \ \ \ \ \ \ \ \ \ \ \ \ \ \ \ \ \  \\ 
\ \ \ \ \ \ \ \ \ \ 1\ \ \ \ \ \ \ \ \ \ \ ;\ \ n=0,\ \lambda \neq 0%
\end{array}%
\right. .  \label{poc}
\end{equation}%

In special case $p=2$ and $q=1$ in $\eqref{ghf}$ , we can obtain $_{2}F_{1} $ Gauss hypergeometric function \cite{Slater, Srivastava},
\begin{equation}
_{2}F_{1}\left[ 
\begin{array}{c}
\alpha \ ,\ \beta \\ 
\gamma%
\end{array}%
;x\right] =\sum_{n=0}^{\infty }\frac{\left( \alpha \right)
	_{n}\left( \beta \right) _{n}}{\left( \gamma \right) _{n}}\frac{x^{n}}{n!},\
\ \ \left\vert x\right\vert <1  \label{hf}
\end{equation}%
where $\alpha ,\ \beta ,\ \gamma ,\ x\in 
\mathbb{C}
\ $and $\gamma \ $neither zero nor a negative integer.

Many elementary functions can be expressed in terms of hypergeometric
functions. Moreover; non-elementary functions that occur in physics and
mathematics have a representation with hypergeometric series. Therefore,
generalizing hypergeometric functions bring with different generalizations
in other disciplines. These generalizations can be made by increasing the number of parameters in
the hypergeometric function or by increasing the number of variables.
Appell, based on the idea that the number of variables can be increased, has
defined Appell hypergeometric functions obtained by multiplying two
hypergeometric functions. These are the four elemanter functions defined in \cite{Slater, Srivastava} 
\begin{eqnarray}
F_{1}\left( \alpha ,\beta ,\beta ^{\prime };\gamma ;x,y\right) 
&=&\sum_{m,n=0}^{\infty }\frac{\left( \alpha \right) _{m+n}\left(
	\beta \right) _{m}\left( \beta ^{\prime }\right) _{n}}{\left( \gamma \right)
	_{m+n}}\frac{x^{m}}{m!}\frac{y^{n}}{n!},  \label{app1} \\
F_{2}\left( \alpha ,\beta ,\beta ^{\prime };\gamma ,\gamma ^{\prime
};x,y\right)  &=&\sum_{m,n=0}^{\infty }\frac{\left( \alpha \right)
	_{m+n}\left( \beta \right) _{m}\left( \beta ^{\prime }\right) _{n}}{\left(
	\gamma \right) _{m}\left( \gamma ^{\prime }\right) _{n}}\frac{x^{m}}{m!}%
\frac{y^{n}}{n!}  \label{app2} \\
F_{3}\left( \alpha ,\alpha ^{\prime },\beta ,\beta ^{\prime };\gamma
;x,y\right)  &=&\sum_{m,n=0}^{\infty }\frac{\left( \alpha \right)
	_{m}\left( \alpha ^{\prime }\right) _{n}\left( \beta \right) _{m}\left(
	\beta ^{\prime }\right) _{n}}{\left( \gamma \right) _{m+n}}\frac{x^{m}}{m!}%
\frac{y^{n}}{n!}  \label{app3} \\
F_{4}\left( \alpha ,\beta ;\gamma ,\gamma ^{\prime };x,y\right) 
&=&\sum_{m,n=0}^{\infty }\frac{\left( \alpha \right) _{m+n}\left(
	\beta \right) _{m+n}}{\left( \gamma \right) _{m}\left( \gamma ^{\prime
	}\right) _{n}}\frac{x^{m}}{m!}\frac{y^{n}}{n!}  \label{app4}
\end{eqnarray}%
where $\left\vert x\right\vert <1,\ \left\vert y\right\vert <1,\ \left\vert
x\right\vert +\left\vert y\right\vert <1,\ \left\vert x\right\vert
<1,\left\vert y\right\vert <1,\ \sqrt{\left\vert x\right\vert }+\sqrt{%
	\left\vert y\right\vert }<1,\ $respectively.

Another generalization of hypergeometric functions is the hypergeometric$\ k$
-function, defined by the Pochhammer $k$-symbol studied by Diaz et al. \cite{Diaz}. This paper includes the $k$-analogue of the Pochhammer symbol and
hypergeometric function, as well as the $k$-generalization of gamma, beta,
and zeta functions with their integral representations and some identities
provided by classical ones. It should be noted that, taking $k=1$ in these
generalizations, the $k$-extensions of the functions reduce to the classical
ones.

Let $k\in 
\mathbb{R}
^{+}$ and $n\in 
\mathbb{N}
^{+}.\ $Hypergeometric $k$-function is defined in \cite{Diaz} as 
\begin{equation}
_{2}F_{1,k}\left[ 
\begin{array}{c}
\alpha \ ,\ \beta  \\ 
\gamma 
\end{array}%
;x\right] :=\ _{2}F_{1,k}\left[ 
\begin{array}{c}
\left( \alpha ,k\right) \ ,\ \left( \beta ,k\right)  \\ 
\left( \gamma ,k\right) 
\end{array}%
;x\right] \ =\sum\limits_{n=0}^{\infty }\frac{\left( \alpha \right)
	_{n,k}\left( \beta \right) _{n,k}}{\left( \gamma \right) _{n,k}}\frac{x^{n}}{%
	n!}
\end{equation}
where $\alpha ,\ \beta ,\ \gamma ,\ x\in 
\mathbb{C}
\ $and $\gamma \ $neither zero nor a negative integer and $\left( \lambda\right) _{n,k}\ $is the Pochhammer $k$-symbol defined in \cite{Diaz} as
\begin{equation}
	\left( \lambda \right) _{n,k}=\left\{ 
	\begin{array}{c}
		\lambda \left( \lambda +k\right) \left( \lambda +2k\right) ...\left( \lambda
		+\left( n-1\right) k\right) \ ;\ \ \ n\geq 1\ \ \ \ \ \ \ \ \ \ \ \ \ \ \ \
		\ \  \\ 
		\ \ \ \ \ \ \ \ \ \ \ \ 1\ \ \ \ \ \ \ \ \ \ \ ;\ \ n=0,\ \lambda \neq 0%
	\end{array}%
	\right. .
\end{equation}

Based on this generalization, Kokologiannaki \cite{Kokolo} obtained different inequalities
and properties for the generalizations of Gamma, Beta and Zeta functions. Some limits with the help of asymptotic properties of $k$-gamma and $k$-beta functions were discussed by Krasniqi \cite{Krasniqi}.
Mubeen et al. \cite{Mubeen} established integral representations of the $k$%
-confluent hypergeometric function and k-hypergeometric function and in
another paper \cite{Mubeen1}, proved the k-analogue of the Kummer's first
formulation using these integral representations. In \cite{Korkmaz}, some families of multilinear and multilateral generating functions for the $k$%
-analogue of the hypergeometric functions were obtained. Studies on this
subject are not limited to these papers, for detailed \cite{Li,Mubeen2,Nisar,Sivamani}.

In \cite{Mubeen5}, Mubeen adapted the $k$-generalization to the Riemann
Liouville fractional integral by using $k$-gamma function. In \cite{Romero}, 
$k$-Riemann Liouville fractional derivative were studied and new properties
were obtained with the help of Fourier and Laplace transforms. In \cite%
{Rahman}, Rahman et al. applied the newly $k$-fractional derivative
operator to $k$-analogue of hypergeometric and Appell functions and obtained
new relations satisfied between them. Furthermore, $k$-fractional derivative
operator was applied to the $k$-Mittag leffler function and the Wright
function.

Our present investigation is motivated by the fact that generalizations of
hypergeometric functions have considerable importance due to their
applications in many disciplines from different perspectives. Therefore, our
study is generally based on the $k$-extension of hypergeometric functions.
The structure of the paper is organized as follows: In section 2, we briefly
give some definitions and preliminary results which are essential in the
following sections as noted in \cite{Diaz,Mubeen5,Mubeen6}. In section 3,
following \cite{Diaz,Mubeen6} and using the same notion, we are concerned
with the $k$-generalizations of $F_{2}$ and $F_{3}$ Appell hypergeometric
functions. Moreover, we prove some main properties such as integral
representations, transformation formulas and some reduction formulas which
enables us to have relations for $k$-hypergeometric functions$\ $and $k$-Appell
functions. In the last part of the paper, applying the theory of Riemann
Liouville $k$-fractional derivative \cite{Rahman} and using the relations
which we consider previous sections, we gain linear and bilinear generating
relations for $k$-analogue of hypergeometric functions and $k$-Appell
functions.

\section{Some\ Definitions\ and\ Preliminary Results\ }

For the sake of completeness, it will be better to examine the preliminary
section in three subsections by reason of the number of theorems and
definitions. In these subsections, we will present some definitions,
properties and results which we need in our investigation in further
sections. We begin by introducing $k$-gamma,\ $k$-beta and $k$-analogue of
hypergeometric function and we continue definition of $k$-generalization of $%
F_{1}$ which is the first Appell function. We conclude this section with
recalling Riemann Liouville fractional derivative,\ $k$-generalization of
this fractional derivative and some important theorems which will be
required in our studies.

Through this paper, we denote by $%
\mathbb{C}
,\ 
\mathbb{R}
,\ 
\mathbb{R}
^{+}$,\ $%
\mathbb{N}
\ $and $%
\mathbb{N}
^{+}$ the sets of complex numbers, real numbers, real and positive numbers
and positive integers with zero and positive integers, respectively.

\subsection{$k$-Generalizations of Gamma, Beta and Hypergeometric Functions}

In this subsection, we will present the definitions of $k$-gamma and $k$-beta
functions are presented and some elemental relations provided by these functions are
introduced by Diaz et al.\ \cite{Diaz} and Mubeen et al.\ \cite{Mubeen2}.
Furthermore, we continue the definition of $k$-hypergeometric function and
we present integral representation and some formulas satisfied from this
generalization \cite{Mubeen, Mubeen1}. 

\begin{definition}
	
	For $x\in 
	\mathbb{C}
	$ and $k\in 
	\mathbb{R}
	^{+},\ $the integral representation of $k$-gamma function $\Gamma _{k}\ $%
	is defined by%
	\begin{equation}
	\Gamma _{k}\left( x\right) =\int\limits_{0}^{\infty }t^{x-1}e^{-\frac{t^{k}%
		}{k}}dt  \label{kg}
	\end{equation}%
	where $\Re\left( x\right) >0$ \cite{Diaz, Mubeen2}.
\end{definition}

\begin{definition}
	For $x,\ y\in 
	\mathbb{C}
	$ and $k\in 
	\mathbb{R}
	^{+},\ $the $k$-beta function $B_{k}\ $ is defined by%
	\begin{equation}
	B_{k}\left( x,y\right) =\frac{1}{k}\int\limits_{0}^{1}t^{\frac{x}{k}%
		-1}\left( 1-t\right) ^{\frac{y}{k}-1}dt  \label{kb}
	\end{equation}%
	where $\Re\left( x\right) >0$ and $\Re\left( y\right) >0$ \cite{Diaz}.
\end{definition}

\begin{proposition}
	Let$\ k\in 
	\mathbb{R}
	^{+},\ a\in 
	\mathbb{R}
	,\ n\in 
	\mathbb{N}
	^{+}.\ $The $k$-gamma function $\Gamma _{k}$ and the $k$-beta function $B_{k}$ satisfy the following properties \cite{Diaz, Mubeen2}, 
	\begin{eqnarray}
	\Gamma _{k}\left( x+k\right) &=&x\Gamma _{k}\left( x\right) ,  \label{kg1} \\
    \Gamma _{k}\left( x\right) &=&k^{\frac{x}{k}-1}\Gamma \left( \frac{x}{k}
	\right) ,  \label{kg2} \\ 
	B_{k}\left( x,y\right) &=&\frac{\Gamma _{k}\left( x\right) \Gamma _{k}\left(
		y\right) }{\Gamma _{k}\left( x+y\right) },  \label{kb3} \\ 
		B_{k}\left( x,y\right) &=&\frac{1}{k}B\left( \frac{x}{k},\frac{y}{k}\right) .
	\label{kb4}
	\end{eqnarray}
\end{proposition}

\begin{definition}
	Let $x$ $\in 
	\mathbb{C}
	,\ k\in 
	\mathbb{R}
	^{+}\ $and $n\in 
	\mathbb{N}
	^{+}.$ Then the Pochhammer $k$-symbol is defined in \cite{Diaz, Mubeen2} by
	\begin{equation}
	\left( x\right) _{n,k}=x\left( x+k\right) \left( x+2k\right) ...\left(
	x+\left( n-1\right) k\right)  \label{kpoc}
	\end{equation}
	
	In particular we denote $\left( x\right) _{0,k}:=1$.
\end{definition}

\begin{proposition}
	If\ $\alpha \in 
	\mathbb{C}
	$ and $m,n\in 
	\mathbb{N}
	^{+}\ $then for $k\in 
	\mathbb{R}
	^{+},$ we have\ 
	\begin{eqnarray}
	\left( \alpha \right) _{n,k} &=&\frac{\Gamma _{k}\left( \alpha +nk\right) }{%
		\Gamma _{k}\left( \alpha \right) },  \label{kpoc1} \\
	\left( \alpha \right) _{n,k} &=&k^{n}\left( \frac{\alpha }{k}\right) _{n},
	\label{kpoc2} \\
	\left( \alpha \right) _{m+n,k} &=&\left( \alpha \right) _{m,k}\left( \alpha
	+mk\right) _{n,k},  \label{kpoc3} 
	\end{eqnarray}%
	where $\left( \alpha \right) _{n}$ and $\left( \alpha \right) _{n,k}$ denote
	the Pochhammer symbol and Pochhammer $k$-symbol\ respectively \cite{Diaz, Mubeen2}.
\end{proposition}

\begin{proposition}
	For any $\alpha \in 
	\mathbb{C}
	$ and $k\in 
	\mathbb{R}
	^{+}$, the following identity holds%
	\begin{equation}
	\sum\limits_{n=0}^{\infty }\left( \alpha \right) _{n,k}\frac{x^{n}}{n!}%
	=\left( 1-kx\right) ^{-\frac{\alpha }{k}}  \label{kpoc5}
	\end{equation}%
	where $\left\vert x\right\vert <\frac{1}{k}$ \cite{Diaz, Mubeen2}.
\end{proposition}

\begin{theorem}
	Assume that $x\in 
	\mathbb{C}
	,\ k\in 
	\mathbb{R}
	^{+}$ and $\Re\left( \gamma \right) >\Re\left( \beta \right)
	>0,\ $then the integral representation of the $k$-hypergeometric function is
	defined in \cite{Mubeen} as
	\begin{equation}
	_{2}F_{1,k}\left[ 
	\begin{array}{c}
	\alpha \ ,\ \beta \\ 
	\gamma%
	\end{array}%
	;x\right] =\frac{\Gamma _{k}\left( \gamma \right) }{k\Gamma _{k}\left( \beta
		\right) \Gamma _{k}\left( \gamma -\beta \right) }\int\limits_{0}^{1}t^{%
		\frac{\beta }{k}-1}\left( 1-t\right) ^{\frac{\gamma -\beta }{k}-1}\left(
	1-kxt\right) ^{-\frac{\alpha }{k}}dt.  \label{ikhf}
	\end{equation}
\end{theorem}

For the following theorem, $_{2}F_{1,k}\left[ 
\begin{array}{c}
\left( \alpha ,1\right) \ ,\ \left( \beta ,k\right) \\ 
\left( \gamma ,k\right)%
\end{array}%
;x\right] \ $is the expression of the following form \cite{Mubeen1},%
\begin{equation}
_{2}F_{1,k}^{\ast }\left[ 
\begin{array}{c}
\alpha \ ,\ \beta  \\ 
\gamma 
\end{array}%
;x\right] :=\ _{2}F_{1,k}\left[ 
\begin{array}{c}
\left( \alpha ,1\right) \ ,\ \left( \beta ,k\right)  \\ 
\left( \gamma ,k\right) 
\end{array}%
;x\right] =\sum\limits_{n=0}^{\infty }\dfrac{\left( \alpha \right)
	_{n}\left( \beta \right) _{n,k}}{\left( \gamma \right) _{n,k}}\dfrac{x^{n}}{%
	n!}.
\end{equation}

\begin{theorem} 
	\cite{Mubeen1} Assume that $x\in 
	\mathbb{C}
	,\ k\in 
	\mathbb{R}
	^{+}$ and $Re\left( \gamma -\beta \right) >0,$ then
	\begin{equation}
	_{2}F_{1,k}\left[ 
	\begin{array}{c}
	\left( \alpha ,1\right) \ ,\ \left( \beta ,k\right) \\ 
	\left( \gamma ,k\right)%
	\end{array}%
	;x\right] :\ =\frac{\Gamma _{k}\left( \gamma \right) \Gamma _{k}\left(
		\gamma -\beta -k\alpha \right) }{\Gamma _{k}\left( \gamma -\beta \right)
		\Gamma _{k}\left( \gamma -k\alpha \right) }.  \label{kummer1}
	\end{equation}
\end{theorem}

For the special case $\alpha =-n,$%
\begin{equation}
_{2}F_{1,k}\left[ 
\begin{array}{c}
\left( -n,1\right) \ ,\ \left( \beta ,k\right) \\ 
\left( \gamma ,k\right)%
\end{array}%
;x\right] \ =\frac{\left( \gamma -\beta \right) _{n,k}}{\left( \gamma
	\right) _{n,k}}.  \label{kummer2}
\end{equation}%

\subsection{ $k$-Generalization of the Appell Function $F_{1}\left( \protect\alpha ,\protect\beta ,\protect\beta ^{\prime };\protect\gamma ;x,y\right) $}

Here, we remind the definition of $k$-analogue of $F_{1}$ which is the first Appell function and some identities which are satisfied by it \cite{Mubeen6}.

\begin{definition}
	\cite{Mubeen6} Let$\ k\in 
	\mathbb{R}
	^{+},\ x,y\in 
	\mathbb{C}
	,\ \alpha ,\ \beta ,\ \beta ^{\prime },\ \gamma \in 
	\mathbb{C}
	$ and $n\in 
	\mathbb{N}
	^{+}.\ $Then the $F_{1,k}$ function with the parameters $\alpha ,\ \beta ,\
	\beta ^{\prime },\ \gamma$ is given by 
	\begin{equation}
	F_{1,k}\left( \alpha ,\beta ,\beta ^{\prime };\gamma ;x,y\right)
	=\sum\limits_{m,n=0}^{\infty }\frac{\left( \alpha \right) _{m+n,k}\left(
		\beta \right) _{m,k}\left( \beta ^{\prime }\right) _{n,k}}{\left( \gamma
		\right) _{m+n,k}}\frac{x^{m}}{m!}\frac{y^{n}}{n!}  \label{kapp1}
	\end{equation}%
	where $\gamma \neq 0,-1,-2,...$ and $\left\vert x\right\vert <\frac{1}{k},\
	\left\vert y\right\vert <\frac{1}{k}.$\bigskip
\end{definition}

\begin{theorem}
	\cite{Mubeen6} Assume that $k\in 
	\mathbb{R}
	^{+},\ x,y\in 
	\mathbb{C}
	,\ \Re\left( \gamma \right) >\Re\left( \alpha \right) >0,\ $then
	the integral representation of the $k$-hypergeometric function is as follows
	\begin{eqnarray}
	&&F_{1,k}\left( \alpha ,\beta ,\beta ^{\prime };\gamma ;x,y\right)   \notag
	\\
	&=&\tfrac{\Gamma _{k}\left( \gamma \right) }{k\Gamma _{k}\left( \alpha
		\right) \Gamma _{k}\left( \gamma -\alpha \right) }\int\limits_{0}^{1}t^{%
		\frac{\alpha }{k}-1}\left( 1-t\right) ^{\frac{\gamma -\alpha }{k}-1}\left(
	1-kxt\right) ^{-\frac{\beta }{k}}\left( 1-kyt\right) ^{-\frac{\beta ^{\prime
		}}{k}}dt.  \label{ikapp}
	\end{eqnarray}
\end{theorem}

\subsection{The Riemann Liouville $k$-Fractional Derivative Operator}

Fractional calculus and its applications have been intensively investigated
for a long time by many researches in numerous disciplines and its attention
has grown tremendously. By making use of the concept of the fractional
derivatives and integrals, various extensions of them has been introduced
and authors have gained different perspectives in many areas such as
engineering, physics, economics, biology, statistics \cite{ Fernandez, Ozarslan}. One of the
generalization of fractional derivatives is Riemann Liouville $k$-fractional
derivative operator studied in \cite{ Rahman, Romero, Azam}.
Here, we remind the definition of Riemann Liouville fractional derivative
and its $k$-generalization and also some theorems which will be used in
further section, are shown.

\begin{definition}
	\cite{Srivastava} The well known Riemann Liouville fractional derivative of order $\mu $ is
	described,\ for a function $f,\ $as follows
	\begin{equation}
	\mathcal{D}_{z}^{\mu }\left\{ f\left( z\right) \right\} =\frac{1}{\Gamma
		\left( -\mu \right) }\int\limits_{0}^{z}f\left( t\right) \left( z-t\right)
	^{-\mu -1}dt  \label{rl1}
	\end{equation}%
	where $\Re\left( \mu \right) <0.$
	
	In particular, for the case $m-1<\Re\left( \mu \right) <m$ where $m=1,2,...$  \eqref{rl1} is written by
	\begin{eqnarray}
	\mathcal{D}_{z}^{\mu }\left\{ f\left( z\right) \right\} &=&\frac{d^{m}}{%
		dz^{m}}\mathcal{D}_{z}^{\mu -m}\left\{ f\left( z\right) \right\}  \label{rl2}
	\\
	&=&\frac{d^{m}}{dz^{m}}\left\{ \frac{1}{\Gamma \left( -\mu +m\right) }%
	\int\limits_{0}^{x}f\left( t\right) \left( z-t\right) ^{-\mu +m-1}dt\right\}
	\notag
	\end{eqnarray}
\end{definition}

\begin{definition}
	\cite{Rahman} The $k$-analogue of Riemann Liouville fractional derivative of order $\mu $
	is defined by%
	\begin{equation}
	_{k}\mathcal{D}_{z}^{\mu }\left\{ f\left( z\right) \right\} =\frac{1}{%
		k\Gamma _{k}\left( -\mu \right) }\int\limits_{0}^{z}f\left( t\right) \left(
	z-t\right) ^{-\frac{\mu }{k}-1}dt  \label{krl1}
	\end{equation}%
	where $\Re\left( \mu \right) <0\ $and$\ k\in 
	\mathbb{R}
	^{+}.$
	
	In particular, for the case $m-1<\Re\left( \mu \right) <m$ where $%
	m=1,2,...,$ \eqref{krl1} is written by%
	\begin{eqnarray}
	_{k}\mathcal{D}_{z}^{\mu }\left\{ f\left( z\right) \right\} &=&\frac{d^{m}}{%
		dz^{m}}\ _{k}\mathcal{D}_{z}^{\mu -mk}\left\{ f\left( z\right) \right\}
	\label{krl2} \\
	&=&\frac{d^{m}}{dz^{m}}\left\{ \frac{1}{k\Gamma _{k}\left( -\mu +mk\right) }%
	\int\limits_{0}^{z}f\left( t\right) \left( z-t\right) ^{-\frac{\mu }{k}%
		+m-1}dt\right\}  \notag
	\end{eqnarray}
\end{definition}

\begin{theorem}
	\cite{Rahman} Let $k\in 
	\mathbb{R}
	^{+},\ \Re\left( \mu \right) <0.$ Then we have
	\begin{equation}
	_{k}\mathcal{D}_{z}^{\mu }\left\{ z^{\frac{\eta }{k}}\right\} =\frac{z^{%
			\frac{\eta -\mu }{k}}}{\Gamma _{k}\left( -\mu \right) }B_{k}\left( \eta
	+k,-\mu \right) .  \label{krl3}
	\end{equation}
\end{theorem}

\begin{theorem}
	\cite{Rahman} Let $Re\left( \mu \right) >0$ and suppose that the function $f\left(
	z\right) $ is analytic at the origin with its Maclaurin expansion has the
	power series expansion 
	\begin{equation}
	f\left( z\right) =\sum\limits_{n=0}^{\infty }a_{n}z^{n}
	\end{equation}
	where$\ \left\vert z\right\vert <\rho ,\ \rho \in 
	\mathbb{R}
	^{+}.$ Then%
	\begin{equation}
	_{k}\mathcal{D}_{z}^{\mu }\left\{ f\left( z\right) \right\}
	=\sum\limits_{n=0}^{\infty }a_{n}\ _{k}\mathcal{D}_{z}^{\mu }\left\{
	z^{n}\right\} .  \label{krl3a}
	\end{equation}
\end{theorem}

\begin{theorem}
	\cite{Rahman} Let $k\in 
	\mathbb{R}
	^{+},\ \Re\left( \mu \right) >\Re\left( \eta \right) >0$ $.\ $%
	Then the following result holds true 
	\begin{equation}
	_{k}\mathcal{D}_{z}^{\eta -\mu }\left\{ z^{\frac{\eta }{k}-1}\left(
	1-kz\right) ^{-\frac{\beta }{k}}\right\} =\frac{\Gamma _{k}\left( \eta
		\right) }{\Gamma _{k}\left( \mu \right) }z^{\frac{\mu }{k}-1}\ _{2}F_{1,k}%
	\left[ 
	\begin{array}{c}
	\beta \ ,\ \eta \\ 
	\mu%
	\end{array}%
	;z\right] \   \label{krl4}
	\end{equation}%
	where $\left\vert z\right\vert <\frac{1}{k}.$
\end{theorem}

\begin{theorem}
	\cite{Rahman} Let $k\in 
	\mathbb{R}
	^{+}.\ $We have the following result
	\begin{equation}
	_{k}\mathcal{D}_{z}^{\eta -\mu }\left\{ z^{\frac{\eta }{k}-1}\left(
	1-kaz\right) ^{-\frac{\alpha }{k}}\left( 1-kbz\right) ^{-\frac{\beta }{k}%
	}\right\} =\frac{\Gamma _{k}\left( \eta \right) }{\Gamma _{k}\left( \mu
		\right) }z^{\frac{\mu }{k}-1}F_{1,k}\left( \eta ,\alpha ,\beta ;\mu
	;az,bz\right)  \label{krl5}
	\end{equation}
	where $\Re\left( \mu \right) >\Re\left( \eta \right) >0,\ \Re\left( \alpha \right) >0,\Re\left( \beta \right) >0$ and $\max
	\left\{ \left\vert az\right\vert ,\left\vert bz\right\vert \right\} <\frac{1%
	}{k}.$
\end{theorem}
\section{$k$-Generalizations of the Appell Functions and Some Transformation Formulas}

In 2015, $k$-generalization of $F_{1}$ Appell function was introduced and
contiguous function relations and integral representation\ of this function
were shown by using the fundamental relations of the Pochhammer $k$-symbol\ 
\cite{Mubeen6}. The $k$-analogue of the $F_{1}\ $was defined but other
Appell $k$-functions such as $F_{2},\ F_{3}$ and $F_{4}$ have not yet been
explored. We now turn our attention the definition of $F_{2}$ and $F_{3}$
and provide the integral representation of them. Also we derive some linear
transformations of Appell functions and give some reduction formulas
involving the $_{2}F_{1,k}$ hypergeometric function.
\begin{definition}
	Let$\ k\in 
	\mathbb{R}
	^{+},\ x,y\in 
	\mathbb{C}
	,\ \alpha ,\beta ,\beta ^{\prime },\gamma $,$\ \gamma ^{\prime }\in 
	\mathbb{C}
	$ and $m,n\in 
	\mathbb{N}
	^{+}.\ $Then the Appell $k$-functions defined by\ 
	\begin{eqnarray}
	F_{2,k}\left( \alpha ,\beta ,\beta ^{\prime };\gamma ,\gamma ^{\prime
	};x,y\right) &=&\sum\limits_{m,n=0}^{\infty }\frac{\left( \alpha \right)
		_{m+n,k}\left( \beta \right) _{m,k}\left( \beta ^{\prime }\right) _{n,k}}{%
		\left( \gamma \right) _{m,k}\left( \gamma ^{\prime }\right) _{n,k}}\frac{%
		x^{m}}{m!}\frac{y^{n}}{n!}  \label{appk2} \\
	&=&\sum\limits_{m=0}^{\infty }\frac{\left( \alpha \right) _{m,k}\left(
		\beta \right) _{m,k}}{\left( \gamma \right) _{m,k}}\ _{2}F_{1,k}\left[ 
	\begin{array}{c}
	\alpha +mk\ ,\ \beta ^{\prime } \\ 
	\gamma ^{\prime }%
	\end{array}%
	;y\right] \frac{x^{m}}{m!}  \notag \\
	F_{3,k}\left( \alpha ,\alpha ^{\prime },\beta ,\beta ^{\prime };\gamma
	;x,y\right) &=&\sum\limits_{m,n=0}^{\infty }\frac{\left( \alpha \right)
		_{m,k}\left( \alpha ^{\prime }\right) _{n,k}\left( \beta \right)
		_{m,k}\left( \beta ^{\prime }\right) _{n,k}}{\left( \gamma \right) _{m+n,k}}%
	\frac{x^{m}}{m!}\frac{y^{n}}{n!}  \label{appk3} \\
	&=&\sum\limits_{m=0}^{\infty }\frac{\left( \alpha \right) _{m,k}\left(
		\beta \right) _{m,k}}{\left( \gamma \right) _{m,k}}\ _{2}F_{1,k}\left[ 
	\begin{array}{c}
	\alpha ^{\prime }\ ,\ \beta ^{\prime } \\ 
	\gamma +mk%
	\end{array}%
	;y\right] \frac{x^{m}}{m!} \notag \\
	F_{4,k}\left( \alpha ,\beta ;\gamma ,\gamma ^{\prime };x,y\right) 
	&=&\sum\limits_{m,n=0}^{\infty }\frac{\left( \alpha \right) _{m+n,k}\left(
		\beta \right) _{m+n,k}}{\left( \gamma \right) _{m,k}\left( \gamma ^{\prime
		}\right) _{n,k}}\frac{x^{m}}{m!}\frac{y^{n}}{n!}  \label{appk4} \\
	&=&\sum\limits_{m=0}^{\infty }\dfrac{\left( \alpha \right) _{m,k}\left(
		\beta \right) _{m,k}}{\left( \gamma \right) _{m,k}}\ _{2}F_{1,k}\left[ 
	\begin{array}{c}
	\alpha +mk\ ,\ \beta +mk \\ 
	\gamma ^{\prime }%
	\end{array}%
	;y\right] \frac{x^{m}}{m!}  \notag
	\end{eqnarray}
	where $\left\vert x\right\vert +\left\vert y\right\vert <\frac{1}{k}, \left\vert x\right\vert <\frac{1}{k},\left\vert y\right\vert <\frac{1}{k}$,	$\sqrt{\left\vert x\right\vert }+\sqrt{\left\vert y\right\vert }<\frac{1}{k}
	\ $ respectively\ and denominators are neither zero and nor negative integers. \\
	Also, the first Appell $k$-function $F_{1,k}\ $defined by \eqref{kapp1} is
	expressed in terms of $_{2}F_{1,k}$ as follows    
	\begin{equation}
	F_{1,k}\left( \alpha ,\beta ,\beta ^{\prime };\gamma ;x,y\right)
	=\sum\limits_{m=0}^{\infty }\dfrac{\left( \alpha \right) _{m,k}\left( \beta
		\right) _{m,k}}{\left( \gamma \right) _{m,k}}\ _{2}F_{1,k}\left[ 
	\begin{array}{c}
	\alpha +mk\ ,\ \beta ^{\prime } \\ 
	\gamma +mk%
	\end{array}%
	;y\right] \frac{x^{m}}{m!}  \label{appk1}
	\end{equation}
\end{definition}

As a first theorem, we consider the integral representation $F_{2,k}$ and $%
F_{3,k}.\ $We note that the integral representation of $F_{1,k}$ can be
found \cite{Mubeen6}.

\begin{theorem}
	Let$\ k\in 
	\mathbb{R}
	^{+}$.\  Integral representations of $F_{2,k}$ and $F_{3,k}$ have the forms of 
	\begin{eqnarray}
	&&F_{2,k}\left( \alpha ,\beta ,\beta ^{\prime };\gamma ,\gamma ^{\prime
	};x,y\right) =\frac{1}{k^{2}B_{k}\left( \beta ,\gamma -\beta \right)
		B_{k}\left( \beta ^{\prime },\gamma ^{\prime }-\beta ^{\prime }\right) } 
	\notag \\
	&&\times \int\limits_{0}^{1}\int\limits_{0}^{1}\frac{t^{\frac{\beta }{k}%
			-1}s^{\frac{\beta ^{\prime }}{k}-1}\left( 1-t\right) ^{\frac{\gamma -\beta }{%
				k}-1}\left( 1-s\right) ^{\frac{\gamma ^{\prime }-\beta ^{\prime }}{k}-1}}{%
		\left( 1-kxt-kys\right) ^{\frac{\alpha }{k}}}dtds  \label{appk5}
	\end{eqnarray}%
	\begin{eqnarray}
	&&F_{3,k}\left( \alpha ,\alpha ^{\prime },\beta ,\beta ^{\prime };\gamma
	;x,y\right) =\frac{\Gamma _{k}\left( \gamma \right) }{k^{2}\Gamma _{k}\left(
		\beta \right) \Gamma _{k}\left( \beta ^{\prime }\right) \Gamma _{k}\left(
		\gamma -\beta -\beta ^{\prime }\right) }  \notag \\
	&&\times \iint\limits_{D}\frac{t^{\frac{\beta }{k}-1}s^{\frac{\beta
				^{\prime }}{k}-1}\left( 1-kxt\right) ^{-\frac{\alpha }{k}}\left(
		1-kys\right) ^{-\frac{\alpha ^{\prime }}{k}}}{\left( 1-t-s\right) ^{1-\frac{%
				\gamma -\beta -\beta ^{\prime }}{k}}}dtds  \label{appk5ab}
	\end{eqnarray}
	where $\Re\left( \gamma \right) >\Re\left( \beta \right) >0$, $\Re\left( \gamma ^{\prime }\right) >\Re\left( \beta ^{\prime}\right) >0 \ and \ D=\left\{ t\geq 0,\ s\geq 0,\ t+s\leq 1\right\}.$
\end{theorem}

\begin{proof}
	From the definition of Pochhammer $k$-symbol, we can write 
	\begin{equation}
	\frac{\left( \beta \right) _{m,k}}{\left( \gamma \right) _{m,k}}=\frac{%
		B_{k}\left( \beta +mk,\gamma -\beta \right) }{B_{k}\left( \beta ,\gamma
		-\beta \right) }\ \ \ \ \ \ \ \ \ \frac{\left( \beta ^{\prime }\right)
		_{n,k}}{\left( \gamma ^{\prime }\right) _{n,k}}=\frac{B_{k}\left( \beta
		^{\prime }+nk,\gamma ^{\prime }-\beta ^{\prime }\right) }{B_{k}\left( \beta
		^{\prime },\gamma ^{\prime }-\beta ^{\prime }\right) }.
	\end{equation}
	We insert these formulas with the integral representation of$\ B_{k}$ into
	the definition of\ $F_{2,k}$ given by \eqref{appk2}, we find that 
	\begin{eqnarray*}
		&&F_{2,k}\left( \alpha ,\beta ,\beta ^{\prime };\gamma ,\gamma ^{\prime
		};x,y\right) =\frac{1}{k^{2}B_{k}\left( \beta ,\gamma -\beta \right)
			B_{k}\left( \beta ^{\prime },\gamma ^{\prime }-\beta ^{\prime }\right) } \\
		&&\times \sum\limits_{m,n=0}^{\infty }\left( \alpha \right) _{m+n,k}\left(
		\int\limits_{0}^{1}t^{\frac{\beta }{k}+m-1}\left( 1-t\right) ^{\frac{\gamma
				-\beta }{k}-1}dt\right) \left( \int\limits_{0}^{1}s^{\frac{\beta ^{\prime }%
			}{k}+n-1}\left( 1-s\right) ^{\frac{\gamma ^{\prime }-\beta ^{\prime }}{k}%
			-1}ds\right) \frac{x^{m}}{m!}\frac{y^{n}}{n!} \\
		&=&\frac{1}{k^{2}B_{k}\left( \beta ,\gamma -\beta \right) B_{k}\left( \beta
			^{\prime },\gamma ^{\prime }-\beta ^{\prime }\right) } \\
		&&\times \int\limits_{0}^{1}\int\limits_{0}^{1}t^{\frac{\beta }{k}-1}s^{%
			\frac{\beta ^{\prime }}{k}-1}\left( 1-t\right) ^{\frac{\gamma -\beta }{k}%
			-1}\left( 1-s\right) ^{\frac{\gamma ^{\prime }-\beta ^{\prime }}{k}%
			-1}\sum\limits_{m,n=0}^{\infty }\left( a\right) _{m+n,k}\frac{\left(
			xt\right) ^{m}}{m!}\frac{\left( ys\right) ^{n}}{n!}dtds \\
		&=&\frac{1}{k^{2}B_{k}\left( \beta ,\gamma -\beta \right) B_{k}\left( \beta
			^{\prime },\gamma ^{\prime }-\beta ^{\prime }\right) } \\
		&&\times \int\limits_{0}^{1}\int\limits_{0}^{1}t^{\frac{\beta }{k}-1}s^{%
			\frac{\beta ^{\prime }}{k}-1}\left( 1-t\right) ^{\frac{\gamma -\beta }{k}%
			-1}\left( 1-s\right) ^{\frac{\gamma ^{\prime }-\beta ^{\prime }}{k}-1}\left(
		1-kxt-kys\right) ^{-\frac{a}{k}}dtds,
	\end{eqnarray*}
	
	which completes the proof.
	
	Formula \eqref{appk5ab} can be proved in a similar way, hence the details are
	omitted.
\end{proof}

\begin{theorem}
	For $k\in 
	\mathbb{R}
	^{+},\ F_{1,k}\ $has the following relation
	\begin{eqnarray}
	&&F_{1,k}\left( \alpha ,\beta ,\beta ^{\prime };\gamma ;x,y\right)   \notag
	\\
	&=&\left( 1-kx\right) ^{-\frac{\beta }{k}}\left( 1-ky\right) ^{-\frac{\beta
			^{\prime }}{k}}F_{1,k}\left( \gamma -\alpha ,\beta ,\beta ^{\prime };\gamma
	;-\frac{x}{1-kx},-\frac{y}{1-ky}\right)   \label{appk7}
	\end{eqnarray}%
	where $\Re\left( \gamma \right) >\Re\left( \alpha \right) >0$
	and $\left\vert \frac{x}{1-kx}\right\vert <\frac{1}{k},\ \left\vert \frac{y}{%
		1-ky}\right\vert <\frac{1}{k},\ \left\vert x\right\vert <\frac{1}{k},\
	\left\vert y\right\vert <\frac{1}{k}.$
\end{theorem}

\begin{proof}
	In \cite{Mubeen6}, the integral representation of\ $F_{1,k}\ $is given by 
	\begin{equation*}
	F_{1,k}\left( \alpha ,\beta ,\beta ^{\prime };\gamma ;x,y\right) =\frac{1}{%
		kB_{k}\left( \alpha ,\gamma -a\right) }\int\limits_{0}^{1}t^{\frac{\alpha }{%
			k}-1}\left( 1-t\right) ^{\frac{\gamma -\alpha }{k}-1}\left( 1-kxt\right) ^{-%
		\frac{\beta }{k}}\left( 1-kyt\right) ^{-\frac{\beta ^{\prime }}{k}}dt.
	\end{equation*}
	
	Performing change of variables $t=1-t_{1}$ in above integral, we can write
	\begin{eqnarray*}
		&&F_{1,k}\left( \alpha ,\beta ,\beta ^{\prime };\gamma ;x,y\right) =\frac{1}{%
			kB_{k}\left( \alpha ,\gamma -a\right) } \\
		&&\times \int\limits_{0}^{1}t_{1}^{\frac{\gamma -\alpha }{k}-1}\left(
		1-t_{1}\right) ^{\frac{\alpha }{k}-1}\left( 1-kx\left( 1-t_{1}\right)
		\right) ^{-\frac{\beta }{k}}\left( 1-ky\left( 1-t_{1}\right) \right) ^{-%
			\frac{\beta ^{\prime }}{k}}dt_{1} \\
		&=&\frac{1}{kB_{k}\left( \alpha ,\gamma -a\right) }\left( 1-kx\right) ^{-%
			\frac{\beta }{k}}\left( 1-ky\right) ^{-\frac{\beta ^{\prime }}{k}} \\
		&&\times \int\limits_{0}^{1}t_{1}^{\frac{\gamma -\alpha }{k}-1}\left(
		1-t_{1}\right) ^{\frac{\alpha }{k}-1}\left( 1+\frac{kxt_{1}}{1-kx}\right) ^{-%
			\frac{\beta }{k}}\left( 1+\frac{kyt_{1}}{1-ky}\right) ^{-\frac{\beta
				^{\prime }}{k}}dt_{1} \\
		&=&\left( 1-kx\right) ^{-\frac{\beta }{k}}\left( 1-ky\right) ^{-\frac{\beta
				^{\prime }}{k}}F_{1,k}\left( \gamma -\alpha ,\beta ,\beta ^{\prime };\gamma
		;-\frac{x}{1-kx},-\frac{y}{1-ky}\right) .
	\end{eqnarray*}
	
	Thus we get the desired result.
\end{proof}

\begin{theorem}
	For $k\in 
	\mathbb{R}
	^{+},\ $we have%
	\begin{eqnarray}
	F_{1,k}\left( \alpha ,\beta ,\beta ^{\prime };\gamma ;x,y\right)=\left(
	1-kx\right) ^{-\frac{\alpha }{k}}F_{1,k}\left( \alpha ,\gamma -\beta -\beta
	^{\prime },\beta ^{\prime };\gamma ;-\tfrac{x}{1-kx},-\tfrac{x-y}{1-kx}%
	\right) ,  \label{appk8} \\
	F_{1,k}\left( \alpha ,\beta ,\beta ^{\prime };\gamma ;x,y\right)=\left(
	1-ky\right) ^{-\frac{\alpha }{k}}F_{1,k}\left( \alpha ,\beta ,\gamma -\beta
	-\beta ^{\prime };\gamma ;-\tfrac{y-x}{1-ky},-\tfrac{y}{1-ky}\right) .
	\label{appk9}
	\end{eqnarray}
\end{theorem}

\begin{proof}
	By a change of variables$\ $using $t=\frac{t_{1}}{1-kx+kt_{1}x}$in the
	integral representation of $F_{1,k},\ $we have that 
	\begin{eqnarray*}
		&&F_{1,k}\left( \alpha ,\beta ,\beta ^{\prime };\gamma ;x,y\right) =\frac{1}{%
			kB_{k}\left( \alpha ,\gamma -\alpha \right) } \\
		&&\times \int\limits_{0}^{1}t^{\frac{\alpha }{k}-1}\left( 1-t\right) ^{%
			\frac{\gamma -\alpha }{k}-1}\left( 1-kxt\right) ^{-\frac{\beta }{k}}\left(
		1-kyt\right) ^{-\frac{\beta ^{\prime }}{k}}dt \\
		&=&\frac{1}{kB_{k}\left( \alpha ,\gamma -\alpha \right) }\left( 1-kx\right)
		^{\frac{\gamma -\alpha -\beta }{k}} \\
		&&\times \int\limits_{0}^{1}t_{1}^{\frac{\alpha }{k}-1}\left(
		1-t_{1}\right) ^{\frac{\gamma -\alpha }{k}-1}\left( 1-kx+kxt_{1}\right) ^{%
			\frac{\beta +\beta ^{\prime }-\gamma }{k}}\left( 1-kx+kxt_{1}-kyt_{1}\right)
		^{-\frac{\beta ^{\prime }}{k}}dt_{1} \\
		&=&\frac{1}{kB_{k}\left( \alpha ,\gamma -\alpha \right) }\left( 1-kx\right)
		^{-\frac{\alpha }{k}} \\
		&&\times \int\limits_{0}^{1}t_{1}^{\frac{\alpha }{k}-1}\left(
		1-t_{1}\right) ^{\frac{\gamma -\alpha }{k}-1}\left( 1+\frac{kxt_{1}}{1-kx}%
		\right) ^{\frac{\beta +\beta ^{\prime }-\gamma }{k}}\left( 1+\frac{%
			kxt_{1}-kyt_{1}}{1-kx}\right) ^{-\frac{\beta ^{\prime }}{k}}dt_{1} \\
		&=&\left( 1-kx\right) ^{-\frac{\alpha }{k}}\ F_{1,k}\left( \alpha ,\gamma
		-\beta -\beta ^{\prime },\beta ^{\prime };\gamma ;-\frac{x}{1-kx},-\frac{x-y%
		}{1-kx}\right) 
	\end{eqnarray*}%
	In the above integral we note that,
	
	Using similar argument with $t=\frac{t_{1}}{1-ky-kt_{1}y}$,\ one can easily
	obtain%
	\begin{equation*}
	F_{1,k}\left( \alpha ,\beta ,\beta ^{\prime };\gamma ;x,y\right) =\left(
	1-ky\right) ^{-\frac{\alpha }{k}}F_{1,k}\left( \alpha ,\beta ,\gamma -\beta
	-\beta ^{\prime };\gamma ;-\frac{y-x}{1-ky},-\frac{y}{1-ky}\right) .
	\end{equation*}
\end{proof}

\begin{theorem}
	\label{TheoremA}Let $k\in 
	\mathbb{R}
	^{+}$ then $F_{1,k}\ $has the following relations,%
	\begin{eqnarray}
	&&F_{1,k}\left( \alpha ,\beta ,\beta ^{\prime };\gamma ;x,y\right)   \notag
	\\
	&=&\left( 1-kx\right) ^{\frac{\gamma -\alpha -\beta }{k}}\left( 1-ky\right)
	^{-\frac{\beta ^{\prime }}{k}}F_{1,k}\left( \gamma -\alpha ,\gamma -\beta
	-\beta ^{\prime },\beta ^{\prime };\gamma ;x,-\tfrac{y-x}{1-ky}\right) 
	\label{appk10} \\
	&&and  \notag \\
	&&F_{1,k}\left( \alpha ,\beta ,\beta ^{\prime };\gamma ;x,y\right)   \notag
	\\
	&=&\left( 1-kx\right) ^{-\frac{\beta }{k}}\left( 1-ky\right) ^{\frac{\gamma
			-\alpha -\beta ^{\prime }}{k}}F_{1,k}\left( \gamma -\alpha ,\beta ,\gamma
	-\beta -\beta ^{\prime };\gamma ;-\tfrac{y-x}{1-kx},y\right).   \label{appk11}
	\end{eqnarray}
\end{theorem}

\begin{proof}
	Using $t=\frac{t_{1}}{1-kx+kxt_{1}}$ and $t_{1}=1-t_{2}\ $in integral representation of $F_{1,k}$, we obtain 
	\begin{eqnarray*}
		&&F_{1,k}\left( \alpha ,\beta ,\beta ^{\prime };\gamma ;x,y\right) =\frac{1}{%
			kB_{k}\left( \alpha ,\gamma -\alpha \right) } \\
		&&\times \int\limits_{0}^{1}t^{\frac{\alpha }{k}-1}\left( 1-t\right) ^{%
			\frac{\gamma -\alpha }{k}-1}\left( 1-kxt\right) ^{-\frac{\beta }{k}}\left(
		1-kyt\right) ^{-\frac{\beta ^{\prime }}{k}}dt \\
		&=&\frac{1}{kB_{k}\left( \alpha ,\gamma -\alpha \right) } \\
		&&\times \int\limits_{0}^{1}t_{2}^{\frac{\gamma -\alpha }{k}-1}\left(
		1-t_{2}\right) ^{\frac{\alpha }{k}-1}\left( 1-kx\right) ^{\frac{\gamma
				-\alpha -\beta }{k}}\left( 1-kxt_{2}\right) ^{\frac{\beta +\beta ^{\prime
				}-\gamma }{k}}\left( 1-ky+kyt_{2}-kxt_{2}\right) ^{-\frac{\beta ^{\prime }}{k%
		}}dt_{2} \\
		&=&\left( 1-kx\right) ^{\frac{\gamma -\alpha -\beta }{k}}\left( 1-ky\right)
		^{-\frac{\beta ^{\prime }}{k}}\frac{1}{kB_{k}\left( \alpha ,\gamma -\alpha
			\right) } \\
		&&\times \int\limits_{0}^{1}t_{2}^{\frac{\gamma -\alpha }{k}-1}\left(
		1-t_{2}\right) ^{\frac{\alpha }{k}-1}\left( 1-kxt_{2}\right) ^{\frac{\beta
				+\beta ^{\prime }-\gamma }{k}}\left( 1+\frac{kyt_{2}-kxt_{2}}{1-ky}\right)
		^{-\frac{\beta ^{\prime }}{k}}dt_{2} \\
		&=&\left( 1-kx\right) ^{\frac{\gamma -\alpha -\beta }{k}}\left( 1-ky\right)
		^{-\frac{\beta ^{\prime }}{k}}F_{1,k}\left( \gamma -\alpha ,\gamma -\beta
		-\beta ^{\prime },\beta ^{\prime };\gamma ;x,-\tfrac{y-x}{1-ky}\right) .
	\end{eqnarray*}%
	Using the same method as above, we can reach \eqref{appk11} easily.
\end{proof}

\begin{theorem}
	Let $k\in 
	\mathbb{R}
	^{+}$, then the following relations hold%
	\begin{eqnarray}
	&&F_{2,k}\left( \alpha ,\beta ,\beta ^{\prime };\gamma ,\gamma ^{\prime
	};x,y\right)   \notag \\
	&=&\left( 1-kx\right) ^{-\frac{\alpha }{k}}F_{2,k}\left( \alpha ,\gamma
	-\beta ,\beta ^{\prime };\gamma ,\gamma ^{\prime };-\frac{x}{1-kx},\frac{y}{%
		1-kx}\right) ,  \label{appk12} \\
	&&  \notag \\
	&&F_{2,k}\left( \alpha ,\beta ,\beta ^{\prime };\gamma ,\gamma ^{\prime
	};x,y\right)   \notag \\
	&=&\left( 1-ky\right) ^{-\frac{\alpha }{k}}F_{2,k}\left( \alpha ,\beta
	,\gamma ^{\prime }-\beta ^{\prime };\gamma ,\gamma ^{\prime };\frac{x}{1-ky}%
	,-\frac{y}{1-ky}\right)  \label{appk13} \\
	&& and \notag \\
	&&F_{2,k}\left( \alpha ,\beta ,\beta ^{\prime };\gamma ,\gamma ^{\prime
	};x,y\right)   \notag \\
	&=&\left( 1-kx-ky\right) ^{-\frac{\alpha }{k}}F_{2,k}\left( \alpha ,\gamma
	-\beta ,\gamma ^{\prime }-\beta ^{\prime };\gamma ,\gamma ^{\prime };-\tfrac{%
		x}{1-kx-ky},-\tfrac{y}{1-kx-ky}\right) .  \label{appk14}
	\end{eqnarray}
	
\end{theorem}

\begin{proof}
	By taking for the first relation $t=1-t_{1},\ $for the second $s=1-s_{1}$
	and finally for the third $t=1-t_{1},\ s=1-s_{1}\ $together in the double
	integral \eqref{appk5}, we find \eqref{appk12}, \eqref{appk13} and \eqref{appk14},
	respectively. These complete the proof.
\end{proof}

We continue with some reduction formulas for Appell functions $F_{1,k}$ and $%
F_{2,k}\ $in terms of the $_{2}F_{1,k}$ generalized hypergeometric function.

\begin{theorem}
	Let $k\in 
	\mathbb{R}
	^{+}$. Then the special cases of $F_{1,k}$ and $F_{2,k}$ are as follows
\begin{eqnarray}
F_{1,k}\left( \alpha ,\beta ,\beta ^{\prime };\gamma ;x,y\right)  &=&\left(
1-kx\right) ^{-\frac{\alpha }{k}}\ _{2}F_{1,k}\left[ 
\begin{array}{c}
\alpha \ ,\ \beta ^{\prime } \\ 
\beta +\beta ^{\prime }%
\end{array}%
;-\frac{x-y}{1-kx}\right],  \label{appk15} \\
F_{1,k}\left( \alpha ,\beta ,\beta ^{\prime };\gamma ;x,y\right)  &=&\left(
1-ky\right) ^{-\frac{\alpha }{k}}\ _{2}F_{1,k}\left[ 
\begin{array}{c}
\alpha \ ,\ \beta  \\ 
\beta +\beta ^{\prime }%
\end{array}%
;-\frac{y-x}{1-ky}\right],  \label{appk16} \\
F_{2,k}\left( \alpha ,\beta ,\beta ^{\prime };\gamma ,\gamma ^{\prime
};x,y\right)  &=&\left( 1-kx\right) ^{-\frac{\alpha }{k}}\ _{2}F_{1,k}\left[ 
\begin{array}{c}
\alpha \ ,\ \beta ^{\prime } \\ 
\gamma ^{\prime }%
\end{array}%
;\frac{y}{1-kx}\right],  \label{appk17} \\
F_{2,k}\left( \alpha ,\beta ,\beta ^{\prime };\gamma ,\gamma ^{\prime
};x,y\right)  &=&\left( 1-ky\right) ^{-\frac{\alpha }{k}}\ _{2}F_{1,k}\left[ 
\begin{array}{c}
\alpha \ ,\ \beta  \\ 
\gamma 
\end{array}%
;\frac{x}{1-ky}\right].  \label{appk18}
\end{eqnarray}
\end{theorem}

\begin{proof}
	Specializing \eqref{appk8} and \eqref{appk9}\ for $\gamma =\beta +\beta ^{\prime
	}\ $and also if we set $\gamma =\beta $ and $\gamma =\beta ^{\prime }$\ in %
	\eqref{appk12} and \eqref{appk13}, we obtain desired results respectively.\ 
\end{proof}

In the next lemma, we will prove Euler transformation for $_{2}F_{1,k}$
hypergeometric function which will be used in the next theorem.

\begin{lemma}
	Let$\ x\in 
	\mathbb{C}
	,\ k\in 
	\mathbb{R}
	^{+}.\ $ Then we have
	\begin{equation}
	_{2}F_{1,k}\left[ 
	\begin{array}{c}
	\alpha \ ,\ \beta \\ 
	\gamma%
	\end{array}%
	;x\right] =\left( 1-kx\right) ^{-\frac{\beta }{k}}\ _{2}F_{1,k}\left[ 
	\begin{array}{c}
	\gamma -\alpha \ ,\ \beta \\ 
	\gamma%
	\end{array}%
	;-\frac{x}{1-kx}\right]  \label{Euler}
	\end{equation}
\end{lemma}

\begin{proof}
	From the definition of $_{2}F_{1,k},$ one gets
	\begin{eqnarray}
	&&\left( 1-kx\right) ^{-\frac{\beta }{k}}\ _{2}F_{1,k}\left[ 
	\begin{array}{c}
	\gamma -\alpha \ ,\ \beta  \\ 
	\gamma 
	\end{array}%
	;-\frac{x}{1-kx}\right]   \notag \\
	&=&\left( 1-kx\right) ^{-\frac{\beta }{k}}\sum\limits_{n=0}^{\infty }\frac{%
		\left( \gamma -\alpha \right) _{n,k}\left( \beta \right) _{n,k}}{\left(
		\gamma \right) _{n,k}}\frac{\left( -1\right) ^{n}x^{n}}{n!\left(
		1-kx\right) ^{n}}  \notag \\
	&=&\sum\limits_{m,n=0}^{\infty }\frac{\left( \gamma -\alpha \right)
		_{n,k}\left( \beta \right) _{n,k}\left( \beta +nk\right) _{m,k}}{\left(
		\gamma \right) _{n,k}}\frac{\left( -1\right) ^{n}x^{m+n}}{n!m!}  \notag \\
	&=&\sum\limits_{m=0}^{\infty }\sum\limits_{n=0}^{m}\frac{\left( \gamma
		-\alpha \right) _{n,k}\left( \beta \right) _{m,k}}{\left( \gamma \right)
		_{n,k}}\frac{\left( -1\right) ^{n}x^{m}}{n!\left( m-n\right) !}
	\label{Euler1}
	\end{eqnarray}%
	Using the identity $\left( m-n\right) !=\frac{\left( -1\right) ^{n}m!}{%
		\left( -m\right) _{n}}\ $in \eqref{Euler1} , we thus find that%
	\begin{eqnarray}
	&&\left( 1-kx\right) ^{-\frac{\beta }{k}}\ _{2}F_{1,k}\left[ 
	\begin{array}{c}
	\gamma -\alpha \ ,\ \beta  \\ 
	\gamma 
	\end{array}%
	;-\frac{x}{1-kx}\right]   \notag \\
	&=&\sum_{m=0}^{\infty }\sum\limits_{n=0}^{m}\frac{\left( \gamma
		-\alpha \right) _{n,k}\left( -m\right) _{n}}{\left( \gamma \right) _{n,k}n!}%
	\frac{\left( \beta \right) _{m,k}x^{m}}{m!}  \notag \\
	&=&\sum_{m=0}^{\infty }\ _{2}F_{1,k}\left[ 
	\begin{array}{c}
	\left( -m,1\right) \ ,\ \left( \gamma -\alpha ,k\right)  \\ 
	\left( \gamma ,k\right) 
	\end{array}%
	;1\right] \left( \beta \right) _{m,k}\frac{x^{m}}{m!}  \label{Euler2}
	\end{eqnarray}%
	Making use of \eqref{kummer2} in \eqref{Euler2}, we get the desired result.
\end{proof}

\begin{theorem}
	Let $k\in 
	\mathbb{R}
	^{+}$. Then we have
	\begin{equation}
	F_{1,k}\left( \alpha ,\beta ,\beta ^{\prime };\gamma ;x,y\right) =\left(
	1-ky\right) ^{-\frac{\beta ^{\prime }}{k}}F_{3,k}\left( \alpha ,\gamma
	-\alpha ,\beta ,\beta ^{\prime };\gamma ;x,-\frac{y}{1-ky}\right)
	\label{appk19}
	\end{equation}
\end{theorem}

\begin{proof}
	Using the definition of $F_{1,k}$ defined by \eqref{appk1} and making use of %
	\eqref{Euler}, we can write%
	\begin{eqnarray*}
		F_{1,k}\left( \alpha ,\beta ,\beta ^{\prime };\gamma ;x,y\right) 
		&=&\sum\limits_{m=0}^{\infty }\frac{\left( \alpha \right) _{m,k}\left(
			\beta \right) _{m,k}}{\left( \gamma \right) _{m,k}}\ _{2}F_{1,k}\left[ 
		\begin{array}{c}
			\alpha +mk\ ,\ \beta ^{\prime } \\ 
			\gamma +mk%
		\end{array}%
		;y\right] \frac{x^{m}}{m!} \\
		&=&\sum\limits_{m=0}^{\infty }\frac{\left( \alpha \right) _{m,k}\left(
			\beta \right) _{m,k}}{\left( \gamma \right) _{m,k}}\left( 1-ky\right) ^{-%
			\frac{\beta ^{\prime }}{k}}\ _{2}F_{1,k}\left[ 
		\begin{array}{c}
			\beta ^{\prime },\ \gamma -\alpha  \\ 
			\gamma +mk%
		\end{array}%
		;-\frac{y}{1-ky}\right] \frac{x^{m}}{m!} \\
		&=&\left( 1-ky\right) ^{-\frac{\beta ^{\prime }}{k}}\sum\limits_{m,n=0}^{%
			\infty }\tfrac{\left( \alpha \right) _{m,k}\left( \beta \right) _{m,k}\left(
			\beta ^{\prime }\right) _{n,k}\left( \gamma -\alpha \right) _{n,k}}{\left(
			\gamma \right) _{m,k}\left( \gamma +mk\right) _{n,k}}\frac{x^{m}}{m!}\tfrac{%
			\left( -\frac{y}{1-ky}\right) ^{n}}{n!} \\
		&=&\left( 1-ky\right) ^{-\frac{\beta ^{\prime }}{k}}F_{3,k}\left( \alpha
		,\gamma -\alpha ,\beta ,\beta ^{\prime };\gamma ;x,-\frac{y}{1-ky}\right) 
	\end{eqnarray*}%
	Thus we finish the proof.
\end{proof}

\section{Generating Relations Involving the Generalized Appell Functions}

In this section, employing the theory of Riemann Liouville $k$-fractional
derivative \cite{Rahman} and making use of the relations which we consider
previous sections, we establish linear and bilinear generating relations for 
$k$-analogue of hypergeometric functions and $k$-Appell functions.

\begin{theorem}
	We have the generating relation%
	\begin{equation}
	\sum\limits_{n=0}^{\infty }\frac{\left( \lambda \right) _{n,k}}{n!}\
	_{2}F_{1,k}\left[ 
	\begin{array}{c}
	\lambda +nk,\ \ \ \ \alpha \\ 
	\beta%
	\end{array}%
	;x\right] t^{n}=\left( 1-kt\right) ^{-\frac{\lambda }{k}}\ _{2}F_{1,k}\left[ 
	\begin{array}{c}
	\lambda ,\ \ \ \ \alpha \\ 
	\beta%
	\end{array}%
	;\frac{x}{1-kt}\right] ,  \label{gf1}
	\end{equation}%
	where $\left\vert x\right\vert <\frac{1}{k}\min \left\{ 1,1-kt\right\} .$
\end{theorem}

\begin{proof}
	To prove the result, consider the elementary identities given by%
	\begin{equation}
	\left( 1-kx-kt\right) ^{-\frac{\lambda }{k}}=\left( 1-kt\right) ^{-\frac{%
			\lambda }{k}}\left( 1-\frac{kx}{1-kt}\right) ^{-\frac{\lambda }{k}},
	\label{gf1a}
	\end{equation}%
	\begin{equation*}
	\left( 1-kx-kt\right) ^{-\frac{\lambda }{k}}=\left( 1-kx\right) ^{-\frac{%
			\lambda }{k}}\left( 1-\frac{kt}{1-kx}\right) ^{-\frac{\lambda }{k}}.
	\end{equation*}%
	From the series expansion using the definition of Pochhammer $k$-symbol\ 
	\cite{Diaz} 
	\begin{equation*}
	\sum\limits_{n=0}^{\infty }\left( \alpha \right) _{n,k}\frac{z^{n}}{n!}%
	=\left( 1-kz\right) ^{-\frac{\alpha }{k}},
	\end{equation*}%
	we can write%
	\begin{eqnarray}
	\left( 1-kx-kt\right) ^{-\frac{\lambda }{k}} &=&\left( 1-kx\right) ^{-\frac{%
			\lambda }{k}}\sum\limits_{n=0}^{\infty }\frac{\left( \lambda \right) _{n,k}%
	}{n!}\left( \frac{t}{1-kx}\right) ^{n}  \notag \\
	&=&\left( 1-kx\right) ^{-\frac{\lambda }{k}}\sum\limits_{n=0}^{\infty }%
	\frac{\left( \lambda \right) _{n,k}}{n!}\left( 1-kx\right) ^{-n}t^{n}  \notag
	\\
	&=&\sum\limits_{n=0}^{\infty }\frac{\left( \lambda \right) _{n,k}}{n!}%
	\left( 1-kx\right) ^{-\frac{\lambda }{k}-n}t^{n}.  \label{gf1b}
	\end{eqnarray}%
	From \eqref{gf1a} and \eqref{gf1b}, we have the equality 
	\begin{equation}
	\sum\limits_{n=0}^{\infty }\frac{\left( \lambda \right) _{n,k}}{n!}\left(
	1-kx\right) ^{-\frac{\lambda }{k}-n}t^{n}=\left( 1-kt\right) ^{-\frac{%
			\lambda }{k}}\left( 1-\frac{kx}{1-kt}\right) ^{-\frac{\lambda }{k}}
	\label{gf1c}
	\end{equation}%
	where $\left\vert t\right\vert <\left\vert 1-kx\right\vert .\ $Multiplying
	both sides of \eqref{gf1c}$\ $by $x^{\frac{\alpha }{k}-1}$and then applying $%
	_{k}D_{x}^{\alpha -\beta }$ to the both sides of \eqref{gf1c}, we can reach%
	\begin{equation*}
	_{k}D_{x}^{\alpha -\beta }\left\{ \sum\limits_{n=0}^{\infty }\frac{\left(
		\lambda \right) _{n,k}}{n!}\ x^{\frac{\alpha }{k}-1}\left( 1-kx\right) ^{-%
		\frac{\lambda }{k}-n}t^{n}\right\} =_{k}D_{x}^{\alpha -\beta }\left\{ \left(
	1-kt\right) ^{-\frac{\lambda }{k}}\ x^{\frac{\alpha }{k}-1}\left( 1-\frac{kx%
	}{1-kt}\right) ^{-\frac{\lambda }{k}}\right\} .
	\end{equation*}%
	Since $\Re(\alpha )>0$ ve $\left\vert t\right\vert <\left\vert
	1-kx\right\vert $, it is possible to change the order of the summation and
	differentiation, we get 
	\begin{eqnarray}
	&&\sum\limits_{n=0}^{\infty }\frac{\left( \lambda \right) _{n,k}}{n!}\
	_{k}D_{x}^{\alpha -\beta }\left\{ x^{\frac{\alpha }{k}-1}\left( 1-kx\right)
	^{-\frac{\lambda }{k}-n}\right\} t^{n}  \label{appk19d} \\
	&=&\left( 1-kt\right) ^{-\frac{\lambda }{k}}\ _{k}D_{x}^{\alpha -\beta
	}\left\{ x^{\frac{\alpha }{k}-1}\left( 1-\frac{kx}{1-kt}\right) ^{-\frac{%
			\lambda }{k}}\right\} .  \notag
	\end{eqnarray}%
	Finally using relation \eqref{krl4} in \eqref{appk19d}, it follows 
	\begin{equation*}
	\sum\limits_{n=0}^{\infty }\frac{\left( \lambda \right) _{n,k}}{n!}\
	_{2}F_{1,k}\left[ 
	\begin{array}{c}
	\lambda +nk,\ \ \ \ \alpha  \\ 
	\beta 
	\end{array}%
	;x\right] t^{n}=\left( 1-kt\right) ^{-\frac{\lambda }{k}}\ _{2}F_{1,k}\left[ 
	\begin{array}{c}
	\lambda ,\ \ \ \ \alpha  \\ 
	\beta 
	\end{array}%
	;\frac{x}{1-kt}\right] 
	\end{equation*}%
	where $\left\vert x\right\vert <\frac{1}{k}\min \left\{ 1,1-kt\right\} .\ $Hence, we
	get the desired result.
\end{proof}
\begin{theorem}
	We have the generating relation%
	\begin{eqnarray}
	&&\sum\limits_{n=0}^{\infty }\frac{\left( \lambda \right) _{n,k}}{n!}\
	_{2}F_{1,k}\left[ 
	\begin{array}{c}
	\rho -nk,\ \ \ \ \alpha  \\ 
	\beta 
	\end{array}%
	;x\right] t^{n}  \notag \\
	&=&\left( 1-kt\right) ^{-\frac{\lambda }{k}}\ F_{1,k}\left[ \alpha ,\rho
	,\lambda ;\beta ;x,-\frac{kxt}{1-kt}\right]   \label{gf2}
	\end{eqnarray}%
	where $\left\vert x\right\vert <\frac{1}{k}, \ \left\vert \frac{kxt}{1-kt}\right\vert <\frac{1}{k}.$
\end{theorem}

\begin{proof}
	Consider the identity%
	\begin{equation}
	\left( 1-k\left( 1-kx\right) t\right) ^{-\frac{\lambda }{k}}=\left(
	1-kt\right) ^{-\frac{\lambda }{k}}\left( 1+\frac{k^{2}xt}{1-kt}\right) ^{-%
		\frac{\lambda }{k}}.  \label{gf2a}
	\end{equation}%
	Under the assumption $\left\vert kt\right\vert <\left\vert 1-kx\right\vert
	^{-1},$we can rewrite \eqref{gf2a} 
	\begin{equation}
	\sum\limits_{n=0}^{\infty }\frac{\left( \lambda \right) _{n,k}}{n!}\left(
	1-kx\right) ^{n}t^{n}=\left( 1-kt\right) ^{-\frac{\lambda }{k}}\left( 1+%
	\frac{k^{2}xt}{1-kt}\right) ^{-\frac{\lambda }{k}}.  \label{gf2b}
	\end{equation}%
	Multiplying $x^{\frac{\alpha }{k}-1}\left( 1-kx\right)^{-\frac{\rho }{k}%
	}$ and taking the $D_{x}^{\alpha -\beta }$ on both sides of \eqref{gf2b}, we
	obtain%
	\begin{eqnarray*}
		&&_{k}D_{x}^{\alpha -\beta }\left\{ \sum\limits_{n=0}^{\infty }\frac{\left(
			\lambda \right) _{n,k}}{n!}\ x^{\frac{\alpha }{k}-1}\left( 1-kx\right) ^{n-%
			\frac{\rho }{k}}t^{n}\right\}  \\
		&=&\ _{k}D_{x}^{\alpha -\beta }\left\{ x^{\frac{\alpha }{k}-1}\left(
		1-kt\right) ^{-\frac{\lambda }{k}}\left( 1-kx\right) ^{-\frac{\rho }{k}%
		}\left( 1+k\frac{kxt}{1-kt}\right) ^{-\frac{\lambda }{k}}\right\} .
	\end{eqnarray*}%
	For $\Re(\alpha )>0$ , interchanging the order of the summation and
	the operator\ $_{k}D_{x}^{\alpha -\beta },\ $we have 
	\begin{eqnarray*}
		&&\sum\limits_{n=0}^{\infty }\frac{\left( \lambda \right) _{n,k}}{n!}\
		_{k}D_{x}^{\alpha -\beta }\left\{ x^{\frac{\alpha }{k}-1}\left( 1-kx\right)
		^{n-\frac{\rho }{k}}\right\} t^{n} \\
		&=&\left( 1-kt\right) ^{-\frac{\lambda }{k}}\ _{k}D_{x}^{\alpha -\beta
		}\left\{ \ x^{\frac{\alpha }{k}-1}\left( 1-kx\right) ^{-\frac{\rho }{k}%
		}\left( 1+k\frac{kxt}{1-kt}\right) ^{-\frac{\lambda }{k}}\right\} .
	\end{eqnarray*}%
	Assuming $\left\vert x\right\vert <\frac{1}{k} \ and \ \left\vert \frac{kxt}{1-kt}\right\vert <\frac{1}{k}$ and using \eqref{krl4} and \eqref{krl5},
	\begin{equation*}
	\sum\limits_{n=0}^{\infty }\frac{\left( \lambda \right) _{n,k}}{n!}\
	_{2}F_{1,k}\left[ 
	\begin{array}{c}
	\rho -nk,\ \ \ \ \alpha  \\ 
	\beta 
	\end{array}%
	;x\right] t^{n}=\left( 1-kt\right) ^{-\frac{\lambda }{k}}\ F_{1,k}\left[
	\alpha ,\rho ,\lambda ;\beta ;x,-\frac{kxt}{1-kt}\right] 
	\end{equation*}%
	the theorem is immediate.
\end{proof}

\begin{theorem}
	We have the generating relations
	\begin{eqnarray}
	&&\sum\limits_{n=0}^{\infty }\frac{\left( \beta -\rho \right) _{n,k}}{n!}\
	_{2}F_{1,k}\left[ 
	\begin{array}{c}
	\rho -nk,\ \ \ \ \alpha  \\ 
	\beta 
	\end{array}%
	;x\right] t^{n}  \notag \\
	&=&\left( 1-kt\right) ^{\frac{\alpha +\rho -\beta }{k}}\ \left(
	1-kt+k^{2}xt\right) ^{-\frac{\alpha }{k}}\ _{2}F_{1,k}\left[ 
	\begin{array}{c}
	\alpha ,\rho  \\ 
	\beta 
	\end{array}%
	;\frac{x}{1-kt+k^{2}xt}\right]   \label{gf3} \\
	&& and \notag \\
	&&\sum\limits_{n=0}^{\infty }\frac{\left( \beta \right) _{n,k}\left( \gamma
		\right) _{n,k}}{\left( \delta \right) _{n,k}n!}\ _{2}F_{1,k}\left[ 
	\begin{array}{c}
	-nk,\ \ \ \ \alpha  \\ 
	\beta 
	\end{array}%
	;x\right] t^{n}  \notag \\
	&=&F_{1,k}\left( \gamma ,\beta -\alpha ,\alpha ;\delta ;t,\left( 1-kx\right)
	t\right)   \label{gf4}
	\end{eqnarray}
\end{theorem}

\begin{proof}
	We use the result of the previous theorem. Setting $\lambda =\beta -\rho $
	in \eqref{gf2}, we find that%
	\begin{equation*}
	\sum\limits_{n=0}^{\infty }\frac{\left( \beta -\rho \right) _{n,k}}{n!}\
	_{2}F_{1,k}\left[ 
	\begin{array}{c}
	\rho -nk,\ \ \ \ \alpha  \\ 
	\beta 
	\end{array}%
	;x\right] t^{n}=\left( 1-kt\right) ^{\frac{\rho -\beta }{k}}\ F_{1,k}\left[
	\alpha ,\rho ,\beta -\rho ;\beta ;x,-\frac{kxt}{1-kt}\right] 
	\end{equation*}%
	If we use reduction formula for $F_{1,k}$ given by \eqref{appk16}, we obtain
	easily the desired result as follows, 
	\begin{eqnarray}
	&&\sum\limits_{n=0}^{\infty }\frac{\left( \beta -\rho \right) _{n,k}}{n!}\
	_{2}F_{1,k}\left[ 
	\begin{array}{c}
	\rho -nk,\ \ \ \ \alpha  \\ 
	\beta 
	\end{array}%
	;x\right] t^{n}  \notag \\
	&=&\left( 1-kt\right) ^{\frac{\alpha +\rho -\beta }{k}}\ \left(
	1-kt+k^{2}xt\right) ^{-\frac{\alpha }{k}}\ _{2}F_{1,k}\left[ 
	\begin{array}{c}
	\alpha ,\rho  \\ 
	\beta 
	\end{array}%
	;\frac{x}{1-kt+k^{2}xt}\right] .  \label{gf3a}
	\end{eqnarray}%
	For $\rho =0$, \eqref{gf3a} gives,%
	\begin{equation}
	\sum\limits_{n=0}^{\infty }\frac{\left( \beta \right) _{n,k}}{n!}\
	_{2}F_{1,k}\left[ 
	\begin{array}{c}
	-nk,\ \ \ \ \alpha  \\ 
	\beta 
	\end{array}%
	;x\right] t^{n}=\left( 1-kt\right) ^{\frac{\alpha -\beta }{k}}\ \left(
	1-kt+k^{2}xt\right) ^{-\frac{\alpha }{k}}.  \label{gf4a}
	\end{equation}%
	Multiplying both sides of \eqref{gf4a}\ with $t^{\frac{\gamma }{k}-1}$ and
	operation of the $_{k}D_{t}^{\gamma -\delta }\ $on\ \eqref{gf4a},\ one can
	easily obtain%
	\begin{eqnarray}
	&&\sum\limits_{n=0}^{\infty }\frac{\left( \beta \right) _{n,k}}{n!}\
	_{2}F_{1,k}\left[ 
	\begin{array}{c}
	-nk,\ \ \ \ \alpha  \\ 
	\beta 
	\end{array}%
	;x\right] _{k}D_{t}^{\gamma -\delta }\left\{ t^{n+\frac{\gamma }{k}%
		-1}\right\}   \notag \\
	&=&D_{t}^{\gamma -\delta }\left\{ t^{\frac{\gamma }{k}-1}\left( 1-kt\right)
	^{\frac{\alpha -\beta }{k}}\ \left( 1-kt+k^{2}xt\right) ^{-\frac{\alpha }{k}%
	}\right\} .  \label{gf4b}
	\end{eqnarray}%
	In view of \eqref{krl3} and \eqref{krl5} on the right and left side of \eqref{gf4b},
	respectively, we can reach 
	\begin{equation*}
	\sum\limits_{n=0}^{\infty }\frac{\left( \beta \right) _{n,k}\left( \gamma
		\right) _{n,k}}{\left( \delta \right) _{n,k} n!}\ _{2}F_{1,k}\left[ 
	\begin{array}{c}
	-nk,\ \ \ \ \alpha  \\ 
	\beta 
	\end{array}%
	;x\right] t^{n}=F_{1,k}\left( \gamma ,\beta -\alpha ,\alpha
	;\delta ;t,\left( 1-kx\right) t\right).
	\end{equation*}
\end{proof}
\begin{theorem}
	We have the generating relation%
	\begin{eqnarray}
	&&\sum\limits_{n=0}^{\infty }\frac{\left( \lambda \right) _{n,k}}{n!}\
	_{2}F_{1,k}\left[ 
	\begin{array}{c}
	\lambda +nk,\ \ \ \ \alpha  \\ 
	\beta 
	\end{array}%
	;x\right] \ _{2}F_{1,k}\left[ 
	\begin{array}{c}
	-nk,\ \ \ \ \gamma  \\ 
	\delta 
	\end{array}%
	;y\right] t^{n}  \notag \\
	&=&\left( 1-kt\right) ^{-\frac{\lambda }{k}}F_{2,k}\left( \lambda ,\alpha
	,\gamma ;\beta ,\delta ;\frac{x}{1-kt},-\frac{kyt}{1-kt}\right) .
	\label{gf5}
	\end{eqnarray}
\end{theorem}

\begin{proof}
	Putting $\left( 1-ky\right) t$ instead of $t\ $in \eqref{gf1}, we can obtain$\ 
	$%
	\begin{eqnarray}
	&&\sum\limits_{n=0}^{\infty }\frac{\left( \lambda \right) _{n,k}}{n!}\
	_{2}F_{1,k}\left[ 
	\begin{array}{c}
	\lambda +nk,\ \ \ \ \alpha  \\ 
	\beta 
	\end{array}%
	;x\right] \left( 1-ky\right) ^{n}t^{n}  \notag \\
	&=&\left( 1-k\left( 1-ky\right) t\right) ^{-\frac{\lambda }{k}}\ _{2}F_{1,k}%
	\left[ 
	\begin{array}{c}
	\lambda ,\ \ \ \ \alpha  \\ 
	\beta 
	\end{array}%
	;\frac{x}{1-k\left( 1-ky\right) t}\right] .  \label{gf5a}
	\end{eqnarray}%
	Multiplying with $y^{\frac{\gamma }{k}-1},\ $employing $_{k}D_{y}^{\gamma
		-\delta }\ $both sides of \eqref{gf5a} and the under the assumption $\Re
	\left( \gamma \right) >0$ interchanging differentiation and summation, we
	can write 
	\begin{eqnarray}
	&&\sum\limits_{n=0}^{\infty }\frac{\left( \lambda \right) _{n,k}}{n!}\
	_{2}F_{1,k}\left[ 
	\begin{array}{c}
	\lambda +nk,\ \ \ \ \alpha  \\ 
	\beta 
	\end{array}%
	;x\right] \ _{k}D_{y}^{\gamma -\delta }\left\{ y^{\frac{\gamma }{k}-1}\left(
	1-ky\right) ^{n}\right\} t^{n}   \label{gf5b} \\
	&=&\ _{k}D_{y}^{\gamma -\delta }\left\{ y^{\frac{\gamma }{k}-1}\left(
	1-k\left( 1-ky\right) t\right) ^{-\frac{\lambda }{k}}\ _{2}F_{1,k}\left[ 
	\begin{array}{c}
	\lambda ,\ \ \ \ \alpha  \\ 
	\beta 
	\end{array}%
	;\tfrac{x}{1-k\left( 1-ky\right) t}\right] \right\} . \notag
	\end{eqnarray}%
	Make use of the formula \eqref{krl4}, we can easily simplify left side of the \eqref{gf5b} as follows, 
	\begin{eqnarray}
	&&\sum\limits_{n=0}^{\infty }\frac{\left( \lambda \right) _{n,k}}{n!}\
	_{2}F_{1,k}\left[ 
	\begin{array}{c}
	\lambda +nk,\ \ \ \ \alpha  \\ 
	\beta 
	\end{array}%
	;x\right] \ _{k}D_{y}^{\gamma -\delta }\left\{ y^{\frac{\gamma }{k}-1}\left(
	1-ky\right) ^{n}\right\} t^{n}   \label{gf5c}\\
	&=&\tfrac{\Gamma _{k}\left( \gamma \right) }{\Gamma _{k}\left( \delta
		\right) }y^{\frac{\delta }{k}-1}\sum\limits_{n=0}^{\infty }\tfrac{\left(
		\lambda \right) _{n,k}}{n!}\ _{2}F_{1,k}\left[ 
	\begin{array}{c}
	\lambda +nk,\ \ \ \ \alpha  \\ 
	\beta 
	\end{array}%
	;x\right] \ _{2}F_{1,k}\left[ 
	\begin{array}{c}
	-nk,\ \ \ \ \gamma  \\ 
	\delta 
	\end{array}%
	;y\right] t^{n}.  \notag 
	\end{eqnarray}
	
	For the right side of the \eqref{gf5b}, using the definition of $_{2}F_{1,k}\ $%
	\ and the formula \eqref{krl3}, one obtain 
	\begin{eqnarray}
	&&_{k}D_{y}^{\gamma -\delta }\left\{ y^{\frac{\gamma }{k}-1}\left( 1-k\left(
	1-ky\right) t\right) ^{-\frac{\lambda }{k}}\ _{2}F_{1,k}\left[ 
	\begin{array}{c}
	\lambda ,\ \ \ \ \alpha  \\ 
	\beta 
	\end{array}%
	;\frac{x}{1-k\left( 1-ky\right) t}\right] \right\}   \notag \\
	&=&\frac{\Gamma _{k}\left( \gamma \right) }{\Gamma _{k}\left( \delta \right) 
	}\left( 1-kt\right) ^{-\frac{\lambda }{k}}y^{\frac{\delta }{k}%
		-1}F_{2,k}\left( \lambda ,\alpha ,\gamma ;\beta ,\delta ;\frac{x}{1-kt},-%
	\frac{kyt}{1-kt}\right)   \label{gf5d}
	\end{eqnarray}%
	where $\left\vert x\right\vert <\frac{1}{k},\ \left\vert y\right\vert <\frac{1}{k}, \ \left\vert \frac{x}{1-kt}\right\vert +\left\vert \frac{kyt}{1-kt}\right\vert <\frac{1}{k},\
	\left\vert \frac{1-ky}{1-x}t\right\vert <\frac{1}{k}$. Combining the
	relations \eqref{gf5c} and \eqref{gf5d}, we get desired result.
\end{proof}

As a special case of \eqref{gf5}, we give the following theorem as follows.

\begin{theorem}
	We have the generating relation%
	\begin{eqnarray}
	&&\sum\limits_{n=0}^{\infty }\tfrac{\left( \beta -\rho \right) _{n,k}}{n!}\
	_{2}F_{1,k}\left[ 
	\begin{array}{c}
	\rho -nk,\ \ \ \ \alpha \\ 
	\beta%
	\end{array}%
	;x\right] \ _{2}F_{1,k}\left[ 
	\begin{array}{c}
	-nk,\ \ \ \ \gamma \\ 
	\delta%
	\end{array}%
	;y\right] t^{n} \label{gf6}  \\
	&=&\left( 1-kx\right) ^{-\frac{\alpha }{k}}\left( 1-kt\right) ^{\frac{\rho
			-\beta }{k}}F_{2,k}\left( \beta -\rho ,\alpha ,\gamma ;\beta ,\delta ;-\tfrac{
		x}{\left( 1-kx\right) \left( 1-kt\right) },-\tfrac{kyt}{1-kt}\right)
	\notag
	\end{eqnarray}
\end{theorem}

\begin{proof}
	For $\lambda =\beta -\rho \ $in \eqref{gf5}, we get 
	\begin{eqnarray*}
		&&\sum\limits_{n=0}^{\infty }\frac{\left( \beta -\rho \right) _{n,k}}{n!}\
		_{2}F_{1,k}\left[ 
		\begin{array}{c}
			\beta -\rho +nk,\ \ \ \ \alpha  \\ 
			\beta 
		\end{array}%
		;x\right] \ _{2}F_{1,k}\left[ 
		\begin{array}{c}
			-nk,\ \ \ \ \gamma  \\ 
			\delta 
		\end{array}%
		;y\right] t^{n} \\
		&=&\left( 1-kt\right) ^{\frac{\rho -\beta }{k}}F_{2,k}\left( \beta -\rho
		,\alpha ,\gamma ;\beta ,\delta ;\frac{x}{1-kt},-\frac{kyt}{1-kt}\right) .
	\end{eqnarray*}%
	Using Euler transformation given by \eqref{Euler} for$\ _{2}F_{1,k}$%
	\begin{eqnarray*}
		&&\sum\limits_{n=0}^{\infty }\frac{\left( \beta -\rho \right) _{n,k}}{n!}%
		\left( 1-kx\right) ^{-\frac{\alpha }{k}}\ _{2}F_{1,k}\left[ 
		\begin{array}{c}
			\rho -nk,\ \ \ \ \alpha  \\ 
			\beta 
		\end{array}%
		;-\frac{x}{1-kx}\right] \ _{2}F_{1,k}\left[ 
		\begin{array}{c}
			-nk,\ \ \ \ \gamma  \\ 
			\delta 
		\end{array}%
		;y\right] t^{n} \\
		&=&\left( 1-kt\right) ^{\frac{\rho -\beta }{k}}F_{2,k}\left( \beta -\rho
		,\alpha ,\gamma ;\beta ,\delta ;\frac{x}{1-kt},-\frac{kyt}{1-kt}\right) 
	\end{eqnarray*}
	
	and putting $-\frac{x}{1-kx}$ instead of $x,\ $we reach the desired result.
\end{proof}

\begin{theorem}
	We have the generating relation%
	\begin{eqnarray}
	&&\sum\limits_{n=0}^{\infty }\frac{\left( \lambda \right) _{n,k}}{n!}\
	_{2}F_{1,k}\left[ 
	\begin{array}{c}
	\lambda +nk,\ \ \ \ \alpha  \\ 
	\beta 
	\end{array}%
	;x\right] \ _{2}F_{1,k}\left[ 
	\begin{array}{c}
	\lambda +nk,\ \ \ \ \gamma  \\ 
	\delta 
	\end{array}%
	;y\right] t^{n}  \notag \\
	&=&\left( 1-kt\right) ^{-\frac{\lambda }{k}}\sum\limits_{n=0}^{\infty }%
	\frac{\left( \lambda \right) _{n,k}\left( \alpha \right) _{n,k}}{\left(
		\beta \right) _{n,k}n!}\left( -\frac{kxy}{1-kt}\right) ^{n}  \label{gf7} \\
	&&\times F_{2,k}\left( \lambda +nk,\alpha +nk,\gamma +nk;\beta +nk,\delta
	+nk;\frac{x}{1-kt},-\frac{ky}{1-kt}\right)   \notag
	\end{eqnarray}
\end{theorem}

\begin{proof}
	Replacing $t$\ by $\frac{t}{1-ky}$ and after some simplification in \eqref{gf1}%
	, we find that%
	\begin{eqnarray*}
		&&\sum\limits_{n=0}^{\infty }\frac{\left( \lambda \right) _{n,k}}{n!}\
		_{2}F_{1,k}\left[ 
		\begin{array}{c}
			\lambda +nk,\ \ \ \ \alpha  \\ 
			\beta 
		\end{array}%
		;x\right] \frac{t^{n}}{\left( 1-ky\right) ^{n+\frac{\lambda }{k}}} \\
		&=&\left( 1-kt\right) ^{-\frac{\lambda }{k}}\sum\limits_{n=0}^{\infty }%
		\frac{\left( \lambda \right) _{n,k}\left( \alpha \right) _{n,k}}{\left(
			\beta \right) _{n,k}n!}\left( \frac{x\left( 1-ky\right) }{1-kt}\right)
		^{n}\left( 1-\frac{ky}{1-kt}\right) ^{-n-\frac{\lambda }{k}}
	\end{eqnarray*}%
		Using the binomial expansion $\left( x+y\right)
	^{n}=\sum\limits_{k=0}^{n}\left( 
	\begin{array}{c}
	n \\ 
	k%
	\end{array}%
	\right) x^{k}y^{n-k},$%
	\begin{eqnarray}
	&&\sum\limits_{n=0}^{\infty }\frac{\left( \lambda \right) _{n,k}}{n!}\
	_{2}F_{1,k}\left[ 
	\begin{array}{c}
	\lambda +nk,\ \ \ \ \alpha  \\ 
	\beta 
	\end{array}%
	;x\right] \frac{t^{n}}{\left( 1-ky\right) ^{n+\frac{\lambda }{k}}}  \notag \\
	&=&\left( 1-kt\right) ^{-\frac{\lambda }{k}}  \notag \\
	&&\times \sum\limits_{n=0}^{\infty }\sum\limits_{k_{1}=0}^{n}\frac{\left(
		\lambda \right) _{n,k}\left( \alpha \right) _{n,k}}{\left( \beta \right)
		_{n,k}n!}\left( 
	\begin{array}{c}
	n \\ 
	k_{1}%
	\end{array}%
	\right) \left( -1\right) ^{n-k_{1}}\left( \frac{x}{1-kt}\right)
	^{k_{1}}\left( \frac{xky}{1-kt}\right) ^{n-k_{1}}\left( 1-\frac{ky}{1-kt}%
	\right) ^{-n-\frac{\lambda }{k}}  \notag \\
	&=&\left( 1-kt\right) ^{-\frac{\lambda }{k}}  \notag \\
	&&\times \sum\limits_{n,k_{1}=0}^{\infty }\frac{\left( \lambda \right)
		_{n+k_{1},k}\left( \alpha \right) _{n+k_{1},k}}{\left( \beta \right)
		_{n+k_{1},k}\left( n+k_{1}\right) !}\left( 
	\begin{array}{c}
	n+k_{1} \\ 
	k_{1}%
	\end{array}%
	\right) \left( -1\right) ^{n}\left( \frac{x}{1-kt}\right) ^{k_{1}}\left( 
	\frac{xky}{1-kt}\right) ^{n}\left( 1-\frac{ky}{1-kt}\right) ^{-n-k_{1}-\frac{%
			\lambda }{k}}  \notag 
	\end{eqnarray}
	\begin{eqnarray}
	&=&\left( 1-kt\right) ^{-\frac{\lambda }{k}}  \notag \\
	&&\times \sum\limits_{n=0}^{\infty }\frac{\left( \lambda \right)
		_{n,k}\left( \alpha \right) _{n,k}}{\left( \beta \right) _{n,k}n!}\left( -%
	\frac{xky}{1-kt}\right) ^{n}\left( 1-\frac{ky}{1-kt}\right) ^{-n-\frac{%
			\lambda }{k}}  \notag \\
	&&\times _{2}F_{1,k}\left[ 
	\begin{array}{c}
	\lambda +nk,\ \ \ \ \alpha +nk \\ 
	\beta +nk%
	\end{array}%
	;\frac{\frac{x}{1-kt}}{1-\frac{ky}{1-kt}}\right]   \label{gf7a}
	\end{eqnarray}
	Multiplying $y^{\frac{\gamma }{k}-1},\ $operating $_{k}D_{y}^{\gamma -\delta
	}\ $and applying \eqref{krl3}, \eqref{krl4}\ and \eqref{krl5}\ together both sides
	of the \eqref{gf7a} (in a similar way of proof of the \eqref{gf5}) for $
	\left\vert x\right\vert <\frac{1}{k},\ \left\vert y\right\vert <\frac{1}{k}
	,\left\vert \frac{x}{1-kt}\right\vert +\left\vert \frac{ky}{1-kt}\right\vert
	<\frac{1}{k}$, we complete the proof.
\end{proof}

\begin{theorem}
	We have the generating relation%
	\begin{eqnarray}
	&&\sum\limits_{n=0}^{\infty }\frac{\left( \lambda \right) _{n,k}}{n!}\
	_{2}F_{1,k}\left[ 
	\begin{array}{c}
	\lambda +nk,\ \ \ \ \alpha  \\ 
	\beta 
	\end{array}%
	;x\right] \ _{2}F_{1,k}\left[ 
	\begin{array}{c}
	\lambda +nk,\ \ \ \ \gamma  \\ 
	\delta 
	\end{array}%
	;y\right] t^{n}  \notag \\
	&=&\left( 1-kt\right) ^{-\frac{\lambda }{k}}\sum\limits_{n=0}^{\infty }%
	\frac{\left( \lambda \right) _{n,k}\left( \alpha \right) _{n,k}\left( \gamma
		\right) _{n,k}}{\left( \beta \right) _{n,k}\left( \delta \right) _{n,k}n!}%
	\left( \frac{k^{3}xyt}{\left( 1-kt\right) ^{2}}\right) ^{n}  \label{gf8} \\
	&&\times \ _{2}F_{1,k}\left[ 
	\begin{array}{c}
	\lambda +nk,\ \ \ \ \alpha +nk \\ 
	\beta +nk%
	\end{array}%
	;\frac{x}{1-kt}\right] \ _{2}F_{1,k}\left[ 
	\begin{array}{c}
	\lambda +nk,\ \ \ \ \gamma +nk \\ 
	\delta +nk%
	\end{array}%
	;\frac{y}{1-kt}\right] .  \notag
	\end{eqnarray}%
	For the special case, we have%
	\begin{eqnarray}
	&&\sum\limits_{n=0}^{\infty }\frac{\left( \lambda \right) _{n,k}}{n!}\
	_{2}F_{1,k}\left[ 
	\begin{array}{c}
	\lambda +nk,\ \ \ \ \alpha  \\ 
	\lambda
	\end{array}%
	;x\right] \ _{2}F_{1,k}\left[ 
	\begin{array}{c}
	\lambda +nk,\ \ \ \ \gamma  \\ 
	\lambda 
	\end{array}%
	;y\right] t^{n}  \notag \\
	&=&\left( 1-kt\right) ^{\frac{\gamma +\alpha -\lambda }{k}}\left(
	1-kt-kx\right) ^{-\frac{\alpha }{k}}\left( 1-kt-ky\right) ^{-\frac{\gamma }{k%
	}}\   \notag \\
	&&\times _{2}F_{1,k}\left[ 
	\begin{array}{c}
	\alpha ,\ \ \ \ \gamma  \\ 
	\lambda 
	\end{array}%
	;\frac{k^{3}xyt}{\left( 1-kt-kx\right) \left( 1-kt-ky\right) }\right] .
	\label{gf9}
	\end{eqnarray}
\end{theorem}

\begin{proof}
	From the elementary identity, we find that%
	\begin{equation}
	\left( \left( 1-kx\right) \left( 1-ky\right) -kt\right) ^{-\frac{\lambda }{k}%
	}=\left( 1-kt\right) ^{-\frac{\lambda }{k}}\left( \left( 1-\frac{kx}{1-kt}%
	\right) \left( 1-\frac{ky}{1-kt}\right) -\frac{k^{3}xyt}{\left( 1-kt\right)
		^{2}}\right) ^{-\frac{\lambda }{k}}.  \label{gf8a}
	\end{equation}%
	for $\left\vert \frac{kt}{\left( 1-kx\right) \left( 1-ky\right) }\right\vert
	<\frac{1}{k}\ $and$\ \left\vert \frac{k^{3}xyt}{\left( 1-kt-kx\right) \left(
		1-kt-ky\right) }\right\vert<\frac{1}{k} $. Applying \eqref{kpoc5}\ for the \eqref{gf8a},
	multiplying $x^{\frac{\alpha }{k}-1}y^{\frac{\gamma }{k}-1}$ and taking $\
	_{k}D_{x}^{\alpha -\beta }\ _{k}D_{y}^{\gamma -\delta }\ $together both
	sides of \eqref{gf8a}, we have 
	\begin{eqnarray*}
		&&_{k}D_{x}^{\alpha -\beta }\ D_{y}^{\gamma -\delta }\left\{
		\sum\limits_{n=0}^{\infty }\frac{\left( \lambda \right) _{n,k}}{n!}x^{\frac{%
				\alpha }{k}-1}\left( 1-kx\right) ^{-\frac{\lambda }{k}-n}y^{\frac{\gamma }{k}%
			-1}\left( 1-ky\right) ^{-\frac{\lambda }{k}-n}t^{n}\right\}  \\
		&=&\left( 1-kt\right) ^{-\frac{\lambda }{k}}\  \\
		&&\times _{k}D_{x}^{\alpha -\beta }\ D_{y}^{\gamma -\delta }\left\{
		\sum\limits_{n=0}^{\infty }\frac{\left( \lambda \right) _{n,k}\left(
			k^{3}t\right) ^{n}}{n!\left( 1-kt\right) ^{2n}}x^{\frac{\alpha }{k}%
			+n-1}\left( 1-\frac{kx}{1-kt}\right) ^{-\frac{\lambda }{k}-n}y^{\frac{\gamma 
			}{k}+n-1}\left( 1-\frac{ky}{1-kt}\right) ^{-\frac{\lambda }{k}-n}\right\} .
	\end{eqnarray*}%
	Under the conditions $\Re\left( \alpha \right) >0,\ \Re\left(
	\gamma \right) >0, \ \left\vert x\right\vert <\frac{1}{k},\ \left\vert
	y\right\vert <\frac{1}{k}, \ \left\vert \frac{x}{1-kt}\right\vert <\frac{1}{k}$ and $%
	\left\vert \frac{y}{1-kt}\right\vert <\frac{1}{k},\ $directly from the
	properties \eqref{krl3}, \eqref{krl4}\ and \eqref{krl5}, we can obtain%
	\begin{eqnarray*}
		&&\sum\limits_{n=0}^{\infty }\frac{\left( \lambda \right) _{n,k}}{n!}\
		_{2}F_{1,k}\left[ 
		\begin{array}{c}
			\lambda +nk,\ \ \ \ \alpha  \\ 
			\beta 
		\end{array}%
		;x\right] \ _{2}F_{1,k}\left[ 
		\begin{array}{c}
			\lambda +nk,\ \ \ \ \gamma  \\ 
			\delta 
		\end{array}%
		;y\right] t^{n} \\
		&=&\left( 1-kt\right) ^{-\frac{\lambda }{k}}\sum\limits_{n=0}^{\infty }%
		\frac{\left( \lambda \right) _{n,k}\left( \alpha \right) _{n,k}\left( \gamma
			\right) _{n,k}}{\left( \beta \right) _{n,k}\left( \delta \right) _{n,k}n!}%
		\left( \frac{k^{3}xyt}{\left( 1-kt\right) ^{2}}\right) ^{n}\  \\
		&&\times _{2}F_{1,k}\left[ 
		\begin{array}{c}
			\lambda +nk,\ \ \ \ \alpha +nk \\ 
			\beta +nk
		\end{array}
		;\frac{x}{1-kt}\right] \ _{2}F_{1,k}\left[ 
		\begin{array}{c}
			\lambda +nk,\ \ \ \ \gamma +nk \\ 
			\delta +nk
		\end{array}
		;\frac{y}{1-kt}\right] .
	\end{eqnarray*}
	For the special case, $\beta =\delta =\lambda \ $in \eqref{gf8}, we have,
	\begin{eqnarray*}
		&&\sum\limits_{n=0}^{\infty }\frac{\left( \lambda \right) _{n,k}}{n!}\
		_{2}F_{1,k}\left[ 
		\begin{array}{c}
			\lambda +nk,\ \ \ \ \alpha  \\ 
			\lambda 
		\end{array}%
		;x\right] \ _{2}F_{1,k}\left[ 
		\begin{array}{c}
			\lambda +nk,\ \ \ \ \gamma  \\ 
			\lambda 
		\end{array}%
		;y\right] t^{n} \\
		&=&\left( 1-kt\right) ^{-\frac{\lambda }{k}} \\
		&&\times \sum\limits_{n=0}^{\infty }\frac{\left( \alpha \right)
			_{n,k}\left( \gamma \right) _{n,k}}{\left( \lambda \right) _{n,k}n!}\left( 
		\frac{k^{3}xyt}{\left( 1-kt\right) ^{2}}\right) ^{n}\left( 1-\frac{kx}{1-kt}%
		\right) ^{-\frac{\alpha +nk}{k}}\left( 1-\frac{ky}{1-kt}\right) ^{-\frac{%
				\gamma +nk}{k}} \\
		&=&\left( 1-kt\right) ^{\frac{\gamma +\alpha -\lambda }{k}}\left(
		1-kt-kx\right) ^{-\frac{\alpha }{k}}\left( 1-kt-ky\right) ^{-\frac{\gamma }{k%
		}}\  \\
		&&\times _{2}F_{1,k}\left[ 
		\begin{array}{c}
			\alpha ,\ \ \ \ \gamma  \\ 
			\lambda 
		\end{array}%
		; \frac{k^{3}xyt}{\left( 1-kt-kx\right) \left( 1-kt-ky\right) }\right] .
	\end{eqnarray*}
\end{proof}


\bibliography{mybibfile}

\end{document}